\documentclass[11pt]{amsart}

\usepackage{amsmath}
\usepackage{amsfonts}
\usepackage{amssymb}
\usepackage{amscd}
\usepackage{amsthm}
\usepackage{filecontents}
\usepackage{framed}
\usepackage{fullpage}
\usepackage{graphicx}
\usepackage{latexsym}
\usepackage[numbers]{natbib}  

\usepackage{hyperref}
\hypersetup{colorlinks,linkcolor={red},citecolor={blue},urlcolor={blue}}

\evensidemargin\oddsidemargin

\newtheorem{theorem}{Theorem}[section]
\newtheorem{lemma}[theorem]{Lemma}
\newtheorem{corollary}[theorem]{Corollary}

\theoremstyle{definition}
\newtheorem{definition}[theorem]{Definition}

\newtheorem{remark}[theorem]{Remark}

\newtheorem{conjecture}[theorem]{Conjecture}

\numberwithin{equation}{section}

\newcommand{\CC}{\mathbb C}

\newcommand{\NN}{\mathbb N}
\newcommand{\PP}{\mathbb P}

\newcommand{\ZZ}{\mathbb Z}

\newcommand{\cD}{\mathcal D}

\newcommand{\SO}{\mathop{\mathrm {SO}}\nolimits}
\newcommand{\Sp}{\mathop{\mathrm {Sp}}\nolimits}
\newcommand{\Orth}{\mathop{\null\mathrm {O}}\nolimits}

\newcommand{\Proj}{\operatorname{Proj}}

\newcommand{\ord}{\operatorname{ord}}

\usepackage{longtable}

\usepackage{tikz}
\usetikzlibrary{positioning}
\usetikzlibrary{arrows}

\newenvironment{psmallmatrix}
  {\left(\begin{smallmatrix}}
  {\end{smallmatrix}\right)}

\begin{document}

\title[Projective spaces as orthogonal modular varieties]{Projective spaces as orthogonal modular varieties}

\author{Haowu Wang}

\address{Max-Planck-Institut f\"{u}r Mathematik, Vivatsgasse 7, 53111 Bonn, Germany}

\email{haowu.wangmath@gmail.com}

\author{Brandon Williams}

\address{Fachbereich Mathematik, Technische Universit\"{a}t Darmstadt, 64289 Darmstadt, Germany}

\email{bwilliams@mathematik.tu-darmstadt.de}


\date{\today}


\begin{abstract}
We construct $16$ reflection groups $\Gamma$ acting on symmetric domains $\cD$ of Cartan type IV, for which the graded algebras of modular forms are freely generated by forms of the same weight, and in particular the Satake--Baily--Borel compactification of $\cD / \Gamma$ is isomorphic to a projective space. Four of these are previously known results of Freitag--Salvati Manni, Matsumoto, Perna and Runge. In addition we find several new modular groups of orthogonal type whose algebras of modular forms are freely generated.
\end{abstract}

\maketitle

\section{Introduction}
In this paper we realize the projective spaces $\PP^3(\CC)$ and $\PP^4(\CC)$ in several ways as orthogonal modular varieties using the theory of Borcherds products. Let $\cD_n$ be a type IV Hermitian symmetric domain of dimension $n$ with $n\geq 3$, i.e. $\cD_n\cong \Orth^+_{n,2}/(\SO_n\times \Orth_2)$ where $\Orth^+_{n,2}$ is the orthogonal group that preserves $\cD_n$, and let $\Gamma$ be an arithmetic subgroup of $\Orth^+_{n,2}$. Then $\Gamma$ acts properly discontinuously on $\cD_n$ and the quotient $\cD_n/\Gamma$ is a quasi-projective variety of dimension $n$. Orthogonal modular forms, i.e. automorphic forms on $\cD_n$ for $\Gamma$, are a powerful tool in the study of these varieties.  An orthogonal modular form of weight $k$ is a holomorphic function on the affine cone over $\cD_n$ which has homogenous degree $-k$ and is invariant under $\Gamma$. Orthogonal modular forms of all weights for $\Gamma$ form a graded algebra $M_*(\Gamma)$. By \cite{BB66} the graded algebra $M_*(\Gamma)$ is finitely generated over $\CC$ and the Satake-Baily-Borel compactification  $(\cD_n/\Gamma)^*$ of the modular variety $\cD / \Gamma$ is a projective variety isomorphic to $\Proj(M_*(\Gamma))$. In particular, if $M_*(\Gamma)$ is freely generated by $n+1$ forms of weights $k_1$, $k_2$,..., $k_{n+1}$, then $(\cD_n/\Gamma)^*=\Proj(M_*(\Gamma))$ is a weighted projective space with weights $k_1$, $k_2$,..., $k_{n+1}$.

Free algebras of modular forms are rare. Many of the known examples are related to irreducible root systems as in \cite{WW20}. It is known that the group $\Gamma$ must be generated by reflections if $M_*(\Gamma)$ is free (see \cite{VP89}). (This immediately rules out any modular groups $\Gamma$ except subgroups of $\mathrm{U}(n, 1)$ and of $\Orth(n,2)$. In this paper $\Gamma$ is always an arithmetic subgroup of $\Orth(n, 2)$.)  The first named author found a necessary and sufficient condition for $M_*(\Gamma)$ to be free in \cite{Wan20} and used it to construct 16 new free algebras of modular forms in \cite{Wan20b}. This condition is essentially the existence of a distinguished modular form which vanishes precisely on all mirrors of reflections in $\Gamma$ with multiplicity one and equals the Jacobian of the $n+1$ generators. In this paper, we use that criterion to find reflection groups $\Gamma$ such that $M_*(\Gamma)$ is freely generated by some forms of the same weight. For any such $\Gamma$ the Satake-Baily-Borel compactification of $\cD_n/\Gamma$ is a projective space. 

Modular varieties isomorphic to projective spaces are very exceptional. The authors are aware of only four examples of dimension $n \ge 3$ in the literature. The first was found by Runge in 1993 \cite{Run93}. Runge's theorem states that the algebra of Siegel modular forms of genus $2$ on the level $4$ subgroup $\Gamma_2[2,4]$ is the polynomial algebra in $4$ theta constants of second order, which implies that the corresponding modular variety is isomorphic to $\PP^3(\CC)$. The second example was found by Matsumoto \cite{Mat93} in the same year. Matsumoto proved that the algebra of symmetric Hermitian modular forms of degree 2 over the Gaussian numbers for the principal congruence group of level $1 + i$ is freely generated by five forms of weight 2, which implies that the corresponding modular variety is isomorphic to $\PP^4(\CC)$. A different proof of this was given by Hermann \cite{Her95}.
The third example was constructed by Freitag and Salvati Manni in 2006. They proved in \citep{FM06} that the algebra of symmetric Hermitian modular forms of degree 2 over the Eisenstein integers for the congruence group of level $\sqrt{-3}$ is freely generated by five forms of weight 1 and therefore the modular variety is also isomorphic to $\PP^4(\CC)$. Finally, Perna \cite{Per16} proved that the algebra of Siegel modular forms for a certain congruence subgroup containing $\Gamma_2[2,4]$ is freely generated by the squares of $4$ theta constants of second order.

We will give a simple and largely uniform proof of the above results in the context of orthogonal modular forms. Runge's and Perna's theorems can be interpreted in terms of orthogonal modular forms for the lattice $2U(4)\oplus A_1$. (Here and below, $U$ is an even unimodular lattice of signature $(1,1)$, and $A_1, A_2$ denote the lattice generated by the root system of the same name.) Matsumoto's theorem is treated using the lattice model $2U(2)\oplus 2A_1$, and for this lattice we also find two new free algebra of modular forms for smaller groups. The first of these is freely generated by five forms of weight $1$ and the second is freely generated by four forms of weight $2$ and one form of weight $1$, such that the modular varieties associated to both of these groups are isomorphic to $\PP^4(\CC)$. The theorem of Freitag--Salvati Manni corresponds to the lattice $2U(3)\oplus A_2$, and we also find two new free algebras of modular forms for larger groups related to this lattice. The first of these is freely generated by five forms of weight $2$ and the second is freely generated by four forms of weight $2$ and one form of weight $1$. The modular varieties associated to the two groups are then isomorphic to $\PP^4(\CC)$. In addition, we determine eight new reflection groups whose associated modular varieties are isomorphic to $\PP^3(\CC)$.  Three of them are related to $2U(2)\oplus A_1$ and the weights of generators of the three algebras are $\{2,2,2,2\}$, $\{1,1,1,1\}$ and $\{1,2,2,2\}$ respectively. Also we determine an algebra related to $2U(3)\oplus A_1$ which is generated by four forms of weight $1$, and two algebras related to $U(2)\oplus U(4)\oplus A_1$ which are generated by four forms of weight $1$ and four forms of weight $1/2$ respectively. Finally we obtain two algebras related to $U\oplus U(2)\oplus A_1(2)$ generated by modular forms of weights $\{2,2,2,2\}$ and $\{1,2,2,2\}$ respectively.

Altogether, we will realize $\PP^3(\CC)$ and $\PP^4(\CC)$ in sixteen ways as orthogonal modular varieties associated to seven lattices in Theorems \ref{th:Runge}, \ref{th:Perna}, \ref{th:Matsumoto}, \ref{th:projective4}, \ref{th:projective5}, \ref{th:FSM}, \ref{th:FSMPerna}, \ref{th:projective8}, \ref{th:projective9}, \ref{th:projective10}, \ref{th:projective11}, \ref{th:projective13}, \ref{th:projective14}, \ref{th:projective15}, \ref{th:projective12}, \ref{th:projective12-2}. The generators of these algebras of modular forms are all constructed as Borcherds products \cite{Bor98}, with one exception where we require an additive theta lift. Along the way we find a simple proof of the famous theorem of Igusa \cite{Igu64} that the algebra of Siegel modular forms of genus $2$ on the level $8$ subgroup $\Gamma_2[4,8]$ is generated by the ten theta constants (Theorem \ref{th:Igusa}), as well as its analogue for Hermitian modular forms over the Gaussian integers (Theorem \ref{th:analogue-Igusa}). We also find several other interesting free algebras of modular forms; for example, we will find a tower of eight congruence subgroups of $\Sp_4(\ZZ)$, all of whose algebras of modular forms are described explicitly and six of which are free.

The layout of this paper is as follows. In \S \ref{sec:preliminaries} we recall the necessary and sufficient condition for free algebras mentioned above (and a straightforward generalization to half-integral weight modular forms). In \S \ref{sec:A1}--\ref{sec:U2A1(2)} we study the algebras of modular forms attached to the seven lattices respectively. In the appendices we give some data related to the input forms for certain Borcherds products; the Fourier expansions of the Borcherds products, and the products themselves as SAGE objects, are available in the ancillary material.

At several points in this paper, we work directly with the Fourier expansions of Borcherds products and therefore need a significant number of coefficients of their input forms. These were computed in SAGE \cite{sagemath} using the algorithm outlined in \cite{Wi18}, an implementation of which is available on GitHub. We also used SAGE for certain graph computations, including much of the data in the appendices.

\section{Preliminaries}\label{sec:preliminaries}
Let $M$ be an even lattice of signature $(n,2)$ with $n\geq 3$ and dual lattice $M^\vee$. The type IV symmetric domain $\cD_n=\cD(M)$ is one of the two connected components of the space 
$$
\{ [\mathcal{Z}]\in \PP(M\otimes\CC) : (\mathcal{Z}, \mathcal{Z})=0, (\mathcal{Z},\bar{\mathcal{Z}}) < 0 \}.
$$
We define the affine cone over $\cD(M)$ as
$$
\mathcal{A}(M)=\{ \mathcal{Z} \in M\otimes\CC : [\mathcal{Z}]\in \cD(M) \}.
$$
Let us denote by $\Orth^+(M)$ the integral orthogonal group of $M$ preserving the connected component $\cD(M)$. The discriminant kernel $\widetilde{\Orth}^+(M)$  is the subgroup of $\Orth^+(M)$ which acts trivially on the discriminant form of $M$. Fix a finite index subgroup $\Gamma$ of $\Orth^+(M)$.

\begin{definition}
Let $k$ be a non-negative integer. A \emph{modular form} of weight $k$ and character $\chi: \Gamma\to \CC^*$ for $\Gamma$ is a holomorphic function $F: \mathcal{A}(M)\to \CC$ satisfying
\begin{align*}
F(t\mathcal{Z})&=t^{-k}F(\mathcal{Z}), \quad \forall t \in \CC^*,\\
F(g\mathcal{Z})&=\chi(g)F(\mathcal{Z}), \quad \forall g\in \Gamma.
\end{align*}
\end{definition}
The graded algebra of modular forms of integral weight is denoted by
$$
M_*(\Gamma)=\bigoplus_{k = 0}^{\infty} M_k(\Gamma).
$$

One can realize the symmetric domain $\cD(M)$ as a tube domain at a zero-dimensional cusp. The above modular form can be viewed as an automorphic form on a tube domain for $\Gamma$ with respect to an automorphy factor. Using the tube domain model one can define modular forms of half-integral weight with a multiplier system. We refer to \cite[\S 3.3]{Bru02} for more details.

For any $r\in M^\vee$ of positive norm, the hyperplane
\begin{equation*}
 \cD_r(M)=r^\perp\cap \cD(M)=\{ [\mathcal{Z}]\in \cD(M) : (\mathcal{Z},r)=0\}
\end{equation*}
is called the \emph{rational quadratic divisor} associated to $r$.
The reflection fixing $\cD_r(M)$ is
\begin{equation*}
\sigma_r(x)=x-\frac{2(r,x)}{(r,r)}r,  \quad x\in M.
\end{equation*}
The hyperplane $\cD_r(M)$ is called the \emph{mirror} of $\sigma_r$. For a non-zero vector $r\in M^\vee$ we denote its order in $M^\vee/M$ by $\ord(r)$.
A primitive vector $r\in M^\vee$ of positive norm is called \emph{reflective} if $\sigma_r\in\Orth^+(M)$, or equivalently, if there exists a positive integer $d$ such that $(r,r)=\frac{2}{d}$ and $\ord(r)=d$ or $\frac{d}{2}$. In this case we call $\cD_r(M)$ a \emph{reflective divisor}. A modular form $F$ for $\Gamma<\Orth^+(M)$ is called \emph{reflective} if its divisor is a sum of reflective divisors. $F$ is called \emph{2-reflective} if its divisor is a sum of divisors $\cD_r(M)$ for $r\in M$ with  $(r,r)=2$. We remark that if $r\in M$ is primitive then the reflection $\sigma_r$ belongs to $\widetilde{\Orth}^+(M)$ if and only if $(r,r)=2$.

For convenience, we recall some results of \cite{Wan20} which will be used in this paper.
The modular Jacobian, or Rankin--Cohen--Ibukiyama differential operator, was first introduced in \cite{AI05} for Siegel modular forms.
\begin{theorem}[Theorem 2.5 in \cite{Wan20}]\label{th:Jacobian}
Let $M$ be an even lattice of signature $(n,2)$ with $n\geq 3$, and let $\Gamma<\Orth^+(M)$ be a finite index subgroup. Let $f_i\in M_{k_i}(\Gamma)$ for $1\leq i \leq n+1$. We view $f_i$ as modular forms on the tube domain at a given zero-dimensional cusp and let $z_i$, $1\leq i \leq n$, be coordinates on the tube domain. We define 
$$
J:=J(f_1,...,f_{n+1})=\left\lvert \begin{array}{cccc}
k_1f_1 & k_2f_2 & \cdots & k_{n+1}f_{n+1} \\ 
\frac{\partial f_1}{\partial z_1} & \frac{\partial f_2}{\partial z_1} & \cdots & \frac{\partial f_{n+1}}{\partial z_1} \\ 
\vdots & \vdots & \ddots & \vdots \\ 
\frac{\partial f_1}{\partial z_n} & \frac{\partial f_2}{\partial z_n} & \cdots & \frac{\partial f_{n+1}}{\partial z_n}
\end{array}   \right\rvert.
$$
\begin{enumerate}
\item $J$ is a cusp form of weight $n+\sum_{i=1}^{n+1}k_i$ for $\Gamma$ with the determinant character $\det$.
\item $J$ is not identically zero if and only if the $n+1$ modular forms $f_i$ are algebraically independent over $\CC$.
\item Let $r\in M$. If the reflection $\sigma_r$ belongs to $\Gamma$, then $J$ vanishes on the hyperplane $\cD_r(M)$.
\end{enumerate}
\end{theorem}

If $M_*(\Gamma)$ is a free algebra then the Jacobian of its generators satisfies some remarkable properties:

\begin{theorem}[Theorem 3.5 in \cite{Wan20}]\label{th:freeJacobian}
Assume that $M_*(\Gamma)$ is a free algebra with generators $f_1,...,f_{n+1}$ of weights $k_1,...,k_{n+1}$.
\begin{enumerate}
\item The Jacobian $J=J(f_1,...,f_{n+1})$ is not identically zero and it is a cusp form of weight $n+\sum_{i=1}^{n+1}k_i$ for $\Gamma$ with the character $\det$.

\item The divisor of $J$ is the sum of all mirrors of reflections in $\Gamma$, each with multiplicity $1$. In particular, $J$ is a reflective cusp form.

\item Let $\{ \Gamma \pi_1, ..., \Gamma \pi_s\}$ denote the $\Gamma$-equivalence classes of mirrors of reflections in $\Gamma$. Then for each $1 \le i \le s$ there exists a modular form $J_i$ for $\Gamma$ with trivial character and divisor $\mathrm{div}(J_i) = 2 \Gamma \pi_i$, and $J^2 = \prod_{i=1}^s J_i$. The forms $J_i$ are irreducible in $M_*(\Gamma)$.

\item There exist polynomials $P$, $P_i$, $1\leq i \leq s$, in $n+1$ variables over $\CC$ such that $J^2=P(f_1,...,f_{n+1})$ and $J_i =P_i(f_1,...,f_{n+1})$. Thus $P=\prod_{i=1}^s P_i$ and these $P_i$ are irreducible. 
\end{enumerate}
\end{theorem}

The following sufficient condition for a graded algebra of modular forms to be free will play a vital role in this paper.

\begin{theorem}[Theorem 5.1 in \cite{Wan20}]\label{th:converseJacobian}
Let $\Gamma<\Orth^+(M)$ be a finite index subgroup. Suppose there exist modular forms $f_1,...,f_{n+1}$ with trivial character whose Jacobian $$J = J(f_1,...,f_{n+1})$$ vanishes exactly on the mirrors of reflections in $\Gamma$ with multiplicity one. Then the graded algebra $M_*(\Gamma)$ is freely generated by $f_1,...,f_{n+1}$ and $\Gamma$ is generated by the reflections whose mirrors lie in the divisor of $J$.
\end{theorem}

This theorem also holds for modular forms of half-integral weight and the proof is nearly the same. We replace $M_*(\Gamma)$ with the graded algebra of half-integral weight modular forms with respect to a fixed multiplier system $v$ of weight $1/2$:
$$
M_*(\Gamma, v)=\bigoplus_{k = 0}^{\infty} M_{\frac{k}{2}}(\Gamma, v^k).
$$
\begin{theorem}\label{th:criterion}
Suppose there exist modular forms $f_1,...,f_{n+1} \in M_*(\Gamma, v)$ whose Jacobian $J = J(f_1,...,f_{n+1})$ vanishes exactly on the mirrors $\cD_r(M)$ of reflections $\sigma_r$ in $\Gamma$ that satisfy $v(\sigma_r) = 1$ with multiplicity one. Then $M_*(\Gamma, v)$ is freely generated by $f_i$.
\end{theorem}
\begin{proof}
Suppose that $M_*(\Gamma,v)$ is not generated by $f_i$. Then $\CC[f_1,...,f_{n+1}] \ne M_*(\Gamma, v)$, and we choose a modular form $f_{n+2} \in M_{k_{n+2}}(\Gamma, v^{2k_{n+2}})$ of minimal weight such that $f_{n+2} \notin \CC[f_1,...,f_{n+1}]$. For $1\leq t \leq n+2$ we define $$J_t = J(f_1,..., \hat f_t, ..., f_{n+2})$$ as the Jacobian of the $n+1$ modular forms $f_i$ omitting $f_t$ (so in particular $J=J_{n+2}$). Similarly to Theorem \ref{th:Jacobian}, one can show that the Jacobian $J_t$ is a modular form of weight $k = n + \sum_{i \ne t} k_i$ and multiplier system $v^{2k}\det$ on $\Gamma$, and that $J_t$ vanishes on all mirrors of reflections in $\Gamma$ satisfying $v(\sigma_r)=1$. Therefore the quotient $g_t := J_t/J$ is a holomorphic modular form in $M_*(\Gamma,v)$.

We compute
$$
0 = \mathrm{det} \begin{psmallmatrix} k_1 f_1 & k_2 f_2 & \cdots & k_{n+2} f_{n+2} \\ k_1 f_1 & k_2 f_2 & \cdots & k_{n+2} f_{n+2} \\ \nabla_z f_1 & \nabla_z f_2 & ... & \nabla_z f_{n+2} \end{psmallmatrix} = \sum_{t=1}^{n+2} (-1)^t k_t f_t J_t = \Big( \sum_{t=1}^{n+2} (-1)^t k_t f_t g_t \Big) \cdot J,
$$
and therefore
$$
(-1)^{n+1}k_{n+2}f_{n+2}= \sum_{t=1}^{n+1}(-1)^t k_t f_t g_t
$$
because $g_{n+2}=1$. In particular, each $g_t$ has weight strictly less than that of $f_{n+2}$. By construction of $f_{n+2}$, this implies $g_t \in \CC[f_1,...,f_{n+1}]$, and therefore $f_{n+2} \in \CC[f_1,...,f_{n+1}]$, a contradiction. Therefore $M_*(\Gamma, v)$ is generated by the $f_i$. Since $M_*(\Gamma, v)$ has Krull dimension $n+1$, these generators are algebraically independent.
\end{proof}
Unlike the integral-weight case, we cannot conclude in general that $\Gamma$ is generated by the reflections whose mirrors are contained in the divisor of the Jacobian. However, this does hold if $v(\sigma_r)=1$ for all reflections $\sigma_r \in \Gamma$.

We can investigate the modular variety $\mathcal{D}_n / \Gamma$ using algebras of half-integral weight modular forms because the Proj is unchanged under Veronese embeddings:
$$
\Proj(M_*(\Gamma,v)^{(d)})=\Proj(M_*(\Gamma,v))=\Proj(M_*(\Gamma)^{(d)})=\Proj(M_*(\Gamma))=(\cD_n / \Gamma)^*,
$$
for any $d \in \mathbb{N}$, where
$$
M_*(\Gamma,v)^{(d)}=\bigoplus_{k\in \NN}M_{\frac{dk}{2}}(\Gamma,v^{dk}), \quad M_*(\Gamma)^{(d)}=\bigoplus_{k\in\NN} M_{dk}(\Gamma).
$$

\section{The \texorpdfstring{$2U(4)\oplus A_1$}{} lattice}\label{sec:A1}
In this section we recover Runge's theorem and Perna's theorem mentioned in the introduction. Runge's theorem asserts that the algebra of Siegel modular forms of genus $2$ on the level $4$ subgroup $\Gamma_2[2,4]$ is freely generated by $4$ theta constants of second order. We will find that the Jacobian of  four theta constants of second order is exactly Igusa's cusp form $\Phi_{5,A_1}$ up to a multiple, where $\Phi_{5,A_1}$ is the product of ten even theta constants. 
The Igusa cusp form $\Phi_{5,A_1}$ is a reflective modular form of weight 5 for $\Orth^+(2U\oplus A_1)$ whose divisor is the sum of all $\cD_r$ for primitive vectors $r\in 2U\oplus A_1^\vee$ with $(r,r)=\frac{1}{2}$ and $\ord(r)=2$, each with multiplicity one. In view of the isomorphisms
\begin{equation}\label{eq:group}
\Orth^+(M)=\Orth^+(M^\vee)=\Orth^+(M^\vee(m)),
\end{equation}
we have $\Orth^+(2U\oplus A_1)=\Orth^+(2U(4)\oplus A_1)$ and the function $\Phi_{5,A_1}$ can also be viewed as a 2-reflective modular form of weight 5 for $2U(4)\oplus A_1$ whose divisor is the sum of $\cD_s$ with multiplicity one for primitive vectors $s\in 2U(\frac{1}{4})\oplus A_1^\vee$ with $(s,s)=\frac{1}{2}$ and $\ord(s)=2$. Unlike $\widetilde{\Orth}^+(2U \oplus A_1)$, these account for all 2-reflections in $\widetilde{\Orth}^+(2U(4)\oplus A_1)$. Moreover, the discriminant kernel $\widetilde{\Orth}^+(2U(4)\oplus A_1)$ is a subgroup of $\widetilde{\Orth}^+(2U\oplus A_1)$, and there are many more Borcherds products on $\widetilde{\Orth}^+(2U(4)\oplus A_1)$. For these reasons it is natural to work with the lattice $2U(4)\oplus A_1$.

We will first recall the Borcherds products of singular weight on $2U(4)\oplus A_1$.
\begin{lemma}\label{lem:2U4A1products}
\noindent
\begin{enumerate}
\item There are 70 holomorphic Borcherds products of singular weight $1/2$ on $2U(4)\oplus A_1$. Their product is the Igusa cusp form of weight $35$.
\item Ten of the singular weight products have inputs with principal part $q^{-1/4}(e_v + e_w)$ where $v$ and $w$ have order 2 in $M^\vee / M$.  These are the ten theta constants and their product is $\Phi_{5,A_1}$. We call them the \emph{Type I products} and denote them $\theta_i$, $i=1,2,...,10$.
\item The remaining 60 singular weight products have inputs with principal part $q^{-1/4}(e_u + e_{-u})$, where $u$ has order 4. We call them the \emph{Type II products}.
\end{enumerate}
\end{lemma}
\begin{proof} These facts can be read off of the Fourier coefficients of a weight 5/2 Eisenstein series. This is carried out in detail in section 4 of \cite{OS19}. With respect to the table of section 4 of \cite{OS19}, the principal parts of the input into the Type I products consist of one of ten vectors $v$ from the 7th orbit and one of six vectors $w$ from the 8th orbit. The vector-valued modular forms with principal part $q^{-1/4}e_v$ alone have non-integral Fourier coefficients so they are, strictly speaking, not valid inputs into the Borcherds lift. However, the $q^{-1/4}e_w$ term does not contribute to their divisor.
\end{proof}

Let $\Orth_1(2U(4)\oplus A_1)$ be the subgroup of $\Orth^+(2U(4)\oplus A_1)$ generated by all reflections associated to the divisor of $\Phi_{5,A_1}$. It is obvious that $\Orth_1(2U(4)\oplus A_1)$ is contained in the discriminant kernel $\widetilde{\Orth}^+(2U(4)\oplus A_1)$. 
The 70 products satisfy the following basic properties.
\begin{lemma}\label{lem:2U4A1products2}
\noindent
\begin{enumerate}
\item The 10 Type I products $\theta_i$ are linearly independent over $\CC$. 
\item The 55 forms $\theta_i\theta_j$ are linearly independent over $\CC$.
\item The 60 Type II products are modular forms of weight 1/2 for $\Orth_1(2U(4)\oplus A_1)$, all of which have the same multiplier system which we denote $v_\Theta$. Moreover, $v_\Theta=1$ for all reflections in $\Orth_1(2U(4)\oplus A_1)$.
\end{enumerate}
\end{lemma}
\begin{proof} (1) The Type I products are modular forms on $\widetilde{\Orth}^+(2U(4)\oplus A_1)$. Let $\sigma_i$ be a reflection associated to the divisor of some $\theta_i$. Since $\sigma_i\in \widetilde{\Orth}^+(2U(4)\oplus A_1)$, we obtain $\sigma_i(\theta_i)=-\theta_i$ and $\sigma_i(\theta_j)=\theta_j$ for $j\neq i$. This implies the linear independence of the 10 products.

(2) Using the argument in (1), the linear independence of the 55 functions $\theta_i\theta_j$ reduces to the linear independence of the 10 squares $\theta_i^2$, which can be shown by computing Fourier coefficients.

(3) Consider the quotient of any two products of type II. This is a weight zero meromorphic modular form on $\Orth_1(2U(4)\oplus A_1)$. Since it does not vanish on any divisors contained in $\mathrm{div}\, \Phi_{5, A_1}$, it has trivial character on $\Orth_1(2U(4) \oplus A_1)$, so the type II products have the same multiplier system. From the principal parts of their input forms we see that no type II product vanishes on the divisor $\cD_r$ for any reflection $\sigma_r \in \Orth_1(2U(4) \oplus A_1)$, which forces $v_{\Theta}(\sigma_r) = 1$. We remark that $v_\Theta^2=1$ on $\Orth_1(2U(4)\oplus A_1)$.
\end{proof}

Throughout this section we will consider sets of four type II products $\Theta_1,\Theta_2,\Theta_3,\Theta_4$ that transform under larger modular groups than $\Orth_1(2U(4) \oplus A_1)$. For this the following definition is useful.

\begin{definition}
A \emph{$*$-set (of type $2U(4)\oplus A_1$)} is a set of four Type II products, each of whose input forms is invariant under all reflections associated to the divisors of the three other products.
\end{definition}

In other words, if the type II product $\Theta_i$ has input with principal part $q^{-1/4}(e_{u_i} + e_{-u_i})$ then $\{\Theta_1, \Theta_2, \Theta_3, \Theta_4\}$ forms a $*$-set if and only if $\sigma_{u_i}(u_j) \in \pm u_j + (2U(4) \oplus A_1)$ for $1 \le i, j \le 4$.

$*$-sets have several interesting properties.
\begin{lemma}\label{lem:*-sets}
\noindent
\begin{enumerate}
\item There are exactly 105 $*$-sets.
\item Every product of type II belongs to exactly seven $*$-sets.
\item There are $1320$ pairs of type II products that do not belong to a $*$-set; there are $360$ pairs that belong to exactly one $*$-set; and the remaining  $90$ pairs belong to exactly three $*$-sets.
\item Any four type II products that form a $*$-set are linearly independent over $\CC$.
\end{enumerate}
\end{lemma}
\begin{proof}
Parts (1)-(3) were checked by computer. We constructed a graph with the Type II products as vertices and an edge between two products if and only if their input forms are invariant under the reflections associated to each other's divisors. (See Appendix A for an image.) The $*$-sets are the maximal cliques in this graph.

(4) Let $f_j$, $1\leq j \leq 4$, be a $*$-set and let $\Gamma_*$ be the group generated by $\Orth_1(2U(4)\oplus A_1)$ and the reflections associated to the divisors of the four products, such that these products are modular forms on $\Gamma_*$. Let $\sigma_i$ be a reflection associated to the divisor of $f_i$. Then $\sigma_i(f_i)=-f_i$ and $\sigma_i(f_j)=f_j$ for $j\neq i$, which implies the linear independence.
\end{proof}

The extra structure given by the 90 pairs of type II products in part (3) will be useful later in this section.

\begin{lemma}\label{lem:2U4A1jacobian} There is a $*$-set $\{\Theta_1,\Theta_2,\Theta_3,\Theta_4\}$ for which the Jacobian $J(\Theta_1, \Theta_2, \Theta_3, \Theta_4)$ equals Igusa's cusp form $\Phi_{5, A_1}$ up to a nonzero constant multiple.
\end{lemma}
\begin{proof} In the notation of Appendix A we took the products labelled $\Theta_1,\Theta_2,\Theta_{26},\Theta_{55}$. By computing their Fourier expansions it is straightforward to show that their Jacobian is not identically zero. Since these products are modular forms of weight $1/2$ and multiplier system $v_\Theta$ on $\Orth_1(2U(4)\oplus A_1)$, and $v_\Theta=1$ for all reflections in $\Orth_1(2U(4)\oplus A_1)$, we conclude from Theorem \ref{th:Jacobian} (4) that their Jacobian $J$ vanishes on the mirrors of all reflections in $\Orth_1(2U(4)\oplus A_1)$. It follows that $J/\Phi_{5,A_1}$ is a holomorphic modular form of weight $0$ and therefore constant. 
\end{proof}

By applying Theorem \ref{th:criterion}, we obtain Runge's theorem in the context of orthogonal groups.
\begin{theorem}\label{th:Runge} The type II products span a four-dimensional space over $\CC$. If $\Theta_1,\Theta_2,\Theta_3,\Theta_4$ are any linearly independent Type II products, then they are algebraically independent and generate the algebra of modular forms:
\begin{align*}
M_*(\Orth_1(2U(4)\oplus A_1),v_\Theta) &=\CC[\Theta_1,\Theta_2,\Theta_3,\Theta_4],\\
(\cD_{3}/ \Orth_1(2U(4)\oplus A_1))^* &\cong \PP^3(\CC).
\end{align*}
\end{theorem}

Perna's theorem can be phrased in terms of a larger modular group associated to a $*$-set.
\begin{theorem}\label{th:Perna}
Let $\{\Theta_1,\Theta_2,\Theta_3,\Theta_4\}$ be a $*$-set and let $\Orth_2(2U(4)\oplus A_1)$ be the subgroup generated by $\Orth_1(2U(4)\oplus A_1)$ and the reflections with mirrors in the divisor of $\prod_{j=1}^4\Theta_j$. Then
\begin{align*}
M_*(\Orth_2(2U(4)\oplus A_1)) &=\CC[\Theta_1^2,\Theta_2^2,\Theta_3^2,\Theta_4^2],\\
(\cD_{3}/ \Orth_2(2U(4)\oplus A_1))^* &\cong \PP^3(\CC).
\end{align*}
\end{theorem}
\begin{proof}
By definition of $*$-sets, the forms $\Theta_j^2$ are modular forms of weight $1$ with trivial character on $\Orth_2(2U(4)\oplus A_1)$. By Lemma \ref{lem:2U4A1jacobian} their Jacobian equals $\Phi_{5,A_1}\prod_{j=1}^4\Theta_j$ up to a non-zero multiple. By Theorem \ref{th:converseJacobian}, the graded algebra of modular forms on $\Orth_2(2U(4)\oplus A_1)$ is freely generated by the four squares.
\end{proof}

As another application, we determine an algebra of modular forms generated by ten theta constants and reprove a well-known theorem of Igusa.

The Jacobian $\Phi_{5,A_1}$ factors as $\prod_{i=1}^{10}\theta_i$ in $M_*(\Orth_1(2U(4)\oplus A_1),v_\Theta)$. By Theorem \ref{th:freeJacobian}, this decomposition determines the character group of $\Orth_1(2U(4)\oplus A_1)$. For $1\leq i\leq 10$, let $\chi_i$ be the character of $\Orth_1(2U(4)\oplus A_1)$ defined as the quotient of the multiplier system of $\theta_i$ by $v_\Theta$. These ten basic characters generate the $1024$ characters of $\Orth_1(2U(4)\oplus A_1)$.
Let $\Orth_1'(2U(4)\oplus A_1)$ be the commutator subgroup of $\Orth_1(2U(4)\oplus A_1)$. Then we have
\begin{align*}
M_\frac{k}{2}(\Orth_1'(2U(4)\oplus A_1),v_\Theta^k)&=\bigoplus_{\chi }M_\frac{k}{2}(\Orth_1(2U(4)\oplus A_1),v_\Theta^k \chi),\\
M_{\frac{k}{2}}(\Orth_1(2U(4)\oplus A_1),v_\Theta^k\chi_i)&=M_{\frac{k-1}{2}}(\Orth_1(2U(4)\oplus A_1),v_\Theta^{k-1})\theta_i, \quad 1\leq i \leq 10,
\end{align*} 
the first sum taken over the character group of $\Orth_1(2U(4) \oplus A_1)$. 

Since each $\theta_i^2$ lies in $M_1(\Orth_1(2U(4)\oplus A_1))$ and the ten squares $\theta_i^2$ are linearly independent, we conclude from Theorem \ref{th:Runge} that the ten $\theta_i^2$ form a basis of $M_1(\Orth_1(2U(4)\oplus A_1))$.
We define $\Orth_0$ as the subgroup on which the basic characters coincide:
$$
\Orth_0(2U(4)\oplus A_1)=\{ \gamma\in \Orth_1(2U(4)\oplus A_1) : \chi_1(\gamma)=\cdots=\chi_{10}(\gamma) \}.
$$  
The group $\Orth_0(2U(4)\oplus A_1)$ properly contains $\Orth_1'(2U(4)\oplus A_1)$ because $\sigma_1\cdots\sigma_{10}\in \Orth_0(2U(4)\oplus A_1)$ where $\sigma_i$ is any reflection associated to the divisor of $\theta_i$ for $1\leq i\leq 10$.
We note that 
$$
\theta_i\in M_{\frac{1}{2}}(\Orth_1(2U(4)\oplus A_1),v_\Theta \chi_i), \quad 1\leq i \leq 10.
$$
Let $\chi$ be the common restriction of the $\chi_i$ to $\Orth_0(2U(4)\oplus A_1)$. We define a multiplier system $v_\vartheta$ of weight $1/2$ on $\Orth_0(2U(4)\oplus A_1)$ by
$$
v_\vartheta=v_\Theta \chi.
$$
Then
\begin{equation}\label{eq:1}
\theta_i\in M_{\frac{1}{2}}(\Orth_0(2U(4)\oplus A_1),v_\vartheta), \quad 1\leq i \leq 10
\end{equation}
and
\begin{equation}\label{eq:2}
\Theta_j \in M_{\frac{1}{2}}(\Orth_0(2U(4)\oplus A_1),v_\vartheta\chi), \quad 1\leq j \leq 4.
\end{equation}

Since the quotient group $\Orth_0(2U(4)\oplus A_1)/ \Orth_1'(2U(4)\oplus A_1)$ is abelian, we obtain an eigenspace decomposition
\begin{equation}\label{eq:3}
M_\frac{k}{2}(\Orth_1'(2U(4)\oplus A_1),v_\vartheta^k)=\bigoplus_{\epsilon}M_\frac{k}{2}(\Orth_0(2U(4)\oplus A_1),v_\vartheta^k \epsilon),
\end{equation}
the direct sum taken over the characters of $\Orth_0(2U(4)\oplus A_1)/ \Orth_1'(2U(4)\oplus A_1)$.
Note that $v_\vartheta=v_\Theta$ on $\Orth_1'(2U(4)\oplus A_1)$. Therefore, every modular form in $M_\frac{k}{2}(\Orth_1'(2U(4)\oplus A_1),v_\vartheta^k)$ can be decomposed as a sum
$$
P_0+\sum_{i=1}^4 P_i \Theta_i,
$$
where $P_0$ is a modular form of weight $k/2$, and $P_1$, $P_2$, $P_3$, $P_4$ are modular forms of weight $(k-1)/2$, and $$P_0,...,P_4 \in \mathbb{C}[\theta_i, 1 \leq i \leq 10].$$ Combining \eqref{eq:1}, \eqref{eq:2} and \eqref{eq:3}, we obtain
$$
M_*(\Orth_0(2U(4)\oplus A_1),v_\vartheta)=\CC[\theta_i, 1\leq i\leq 10]
$$
and we see that $\chi$ is the unique non-trivial character of $\Orth_2(2U(4)\oplus A_1)/ \Orth_1'(2U(4)\oplus A_1)$. Altogether, we obtain the generators found by Igusa in \cite{Igu64}:
\begin{theorem}\label{th:Igusa}
\begin{align*}
M_*(\Orth_0(2U(4)\oplus A_1),v_\vartheta)&=\CC[\theta_i, 1\leq i\leq 10],\\
M_*(\Orth_1'(2U(4)\oplus A_1),v_\Theta)&=\CC[\Theta_1,\Theta_2,\Theta_3,\Theta_4, \theta_i, 1\leq i\leq 10].
\end{align*}
\end{theorem}

\begin{remark}\label{rem:decomposotionO2}
The ten theta constants $\theta_i$ are not all modular on $\Orth_2(2U(4)\oplus A_1)$. In fact, $\Phi_{5,A_1}$ factors as the product of five forms of weights $1/2$, $1/2$, $1$, $1$ and $2$ on $\Orth_2(2U(4)\oplus A_1)$, which are products of some theta constants. 
\end{remark}

We can now construct some new free algebras of Siegel modular forms associated to $*$-sets. Let $\Theta_1,...,\Theta_4$ be a $*$-set and let $\Orth_2(2U(4) \oplus A_1)$ be the group generated by its reflections and by $\Orth_1(2U(4) \oplus A_1)$.
\begin{theorem}\label{th:O2multiplier}
For $1\leq i \leq 4$ let $v_i$ be the multiplier system of $\Theta_i$ on $\Orth_2(2U(4)\oplus A_1)$.  
Then
$$
M_*(\Orth_2(2U(4)\oplus A_1),v_i)=\CC[\Theta_i, \Theta_j^2, 1\leq j\leq 4, j\neq i], \quad 1\leq i \leq 4.
$$

\end{theorem}
This includes Theorem \ref{th:Perna} which describes the subring of integer-weight forms.
\begin{proof}
Without loss of generality we take $i=1$. It is clear that $v_1=1$ for all reflections associated to the divisor of $\prod_{j=2}^4\Theta_j$ and  $v_1=-1$ for reflections associated to the divisor of $\Theta_1$. By Lemma \ref{lem:2U4A1jacobian} the Jacobian of $\Theta_1$, $\Theta_j^2$ for $2\leq j\leq 4$ equals $\Phi_{5,A_1}\prod_{j=2}^4\Theta_j$ up to a non-zero multiple. The claim follows from Theorem \ref{th:criterion}. 
\end{proof}

Let $\Orth_{1,1}(2U(4)\oplus A_1)$, $\Orth_{1, 12}(2U(4) \oplus A_1)$ and $\Orth_{1, 123}(2U(4) \oplus A_1)$ denote the subgroups generated by reflections associated to the divisor of $\Phi_{5,A_1}\Theta_1$, $\Phi_{5, A_1} \Theta_1 \Theta_2$ and $\Phi_{5, A_1} \Theta_1 \Theta_2 \Theta_3$, respectively. By a similar argument, we can prove the following results using Theorem \ref{th:criterion}:
\begin{align*}
M_*(\Orth_{1,1}(2U(4)\oplus A_1),v_4)&=\CC[\Theta_1^2,\Theta_2,\Theta_3,\Theta_4],\\
M_*(\Orth_{1,12}(2U(4)\oplus A_1),v_4)&=\CC[\Theta_1^2,\Theta_2^2,\Theta_3,\Theta_4],\\
M_*(\Orth_{1,123}(2U(4)\oplus A_1),v_4)&=\CC[\Theta_1^2,\Theta_2^2,\Theta_3^2,\Theta_4].
\end{align*}

We remark that $\Orth_{1,1}(2U(4)\oplus A_1)$ does not contain any reflections associated to the divisor of $\Theta_2\Theta_3\Theta_4$. Indeed, the multiplier systems $v_2$, $v_3$ and $v_4$ coincide on $\Orth_{1,1}(2U(4)\oplus A_1)$ because the quotient of any two of them defines a character whose values on generators are always 1. If $\Orth_{1,1}(2U(4)\oplus A_1)$ contained a reflection $\sigma$ associated to the divisor of $\Theta_2$ then we would find $v_2(\sigma)=-1$ but $v_3(\sigma)=v_4(\sigma)=1$, a contradiction. Similarly, $\Orth_{1,12}(2U(4)\oplus A_1)$ does not contain any reflections associated to the divisor of $\Theta_3\Theta_4$.

On the other hand, $\Orth_{1,123}(2U(4)\oplus A_1)$ does contain reflections associated to the divisor of $\Theta_4$. In fact,
\begin{equation}\label{eq:O2}
\Orth_{1,123}(2U(4)\oplus A_1)=\Orth_{2}(2U(4)\oplus A_1).
\end{equation}
(Of course, this agrees with Theorem \ref{th:O2multiplier}.) If this was not true, then the only alternative would be $M_*(\Orth_{1,123}(2U(4)\oplus A_1),v_4^2)=\CC[\Theta_i^2, i=1,2,3,4]$, contradicting the decomposition of the Jacobian in Theorem \ref{th:freeJacobian}.

Finally we obtain some larger groups generated by reflections and free algebras of modular forms associated to triples of $*$-sets. Suppose $\Theta_1, \Theta_2$ is one of the 90 pairs of type II products that can be extended to three distinct $*$-sets, denoted $$\{\Theta_1, \Theta_2, \Theta_3, \Theta_4\}, \quad \{\Theta_1, \Theta_2, \Theta_5, \Theta_6\}, \quad \{\Theta_1, \Theta_2, \Theta_7, \Theta_8\}.$$ The principal part into the input function to $\Theta_i$ will be denoted $q^{-1/4} (e_{u_i} + e_{-u_i})$.

In this situation the reflection $\sigma_{u_5}$ fixes the set $\{\pm u_1, \pm u_2, \pm u_3, \pm u_4\}$ and has $u_1$ and $u_2$ as eigenvectors. It cannot have $u_3$ and $u_4$ as eigenvectors as otherwise the proof of Lemma \ref{lem:*-sets} yields five linearly independent type II products. Therefore it maps $u_3$ to $\pm u_4$ and $u_4$ to $\pm u_3$. Taking Borcherds lifts and using the fact that all type II products have Fourier coefficients in $\mathbb{Z}[i]$, this implies $\sigma_{u_5}(\Theta_3) = i^k \Theta_4$ for some $k \in \{0, 1, 2, 3\}$. The same statements hold for the reflections $\sigma_{u_i}$, $i = 6,7,8.$

Let $\Orth_{2,56}(2U(4)\oplus A_1)$ be the group generated by $\Orth_{2}(2U(4)\oplus A_1)$ and the reflections associated to the divisor of $\Theta_5$. Since $\sigma_{u_3}$ swaps $u_5$ and $u_6$ (up to multiples), $\Orth_{2,56}(2U(4) \oplus A_1)$ also contains the reflections associated to the divisor of $\Theta_6$. It is clear that $\Theta_3^2\Theta_4^2\in M_2(\Orth_{2,56}(2U(4)\oplus A_1))$. From the previous paragraph it follows that $\sigma_{u_5}(\Theta_3^2)= c \Theta_4^2$ where $c \in \{\pm 1\}$. Then $\sigma_{u_5}(\Theta_4^2)= c\Theta_3^2$ and therefore $\Theta_3^2 + c\Theta_4^2\in M_1(\Orth_{2,56}(2U(4)\oplus A_1))$.

On the other hand, $\Theta_5\Theta_6\in M_1(\Orth_{2}(2U(4)\oplus A_1))$ is a $\CC$-linear combination of $\Theta_i^2$ for $i=1,2,3,4$. By considering the  action of $\sigma_{u_5}$ on these functions, we find that $\Theta_5\Theta_6$ is equal to $\Theta_3^2 - c\Theta_4^2$ up to a non-zero constant multiple. Using Lemma \ref{lem:2U4A1jacobian}, we see that the Jacobian of $\Theta_1^2$, $\Theta_2^2$, $\Theta_3^2 + c\Theta_4^2$ and $\Theta_3^2\Theta_4^2$ is equal to $\Phi_{5,A_1}\prod_{i=1}^6\Theta_i$ up to a non-zero multiple. We have therefore proved the following:
$$ 
M_*(\Orth_{2,56}(2U(4)\oplus A_1))=\CC[\Theta_1^2,\Theta_2^2,\Theta_3^2 + c\Theta_4^2, \Theta_3^2\Theta_4^2].
$$

Similarly, since $\Theta_7\Theta_8\in M_1(\Orth_{2,56}(2U(4)\oplus A_1))$, it is equal to $\Theta_3^2+c\Theta_4^2$ up to a non-zero constant by considering the action of $\sigma_{u_7}$. Letting $\Orth_{2,78}(2U(4)\oplus A_1)$ be the group generated by $\Orth_{2}(2U(4) \oplus A_1)$ and the reflections associated to the divisor of $\Theta_7$, we obtain by the same argument
$$
M_*(\Orth_{2,78}(2U(4)\oplus A_1)) = \CC[\Theta_1^2, \Theta_2^2, \Theta_3^2 - c\Theta_4^2, \Theta_3^2 \Theta_4^2].
$$

Finally, let $\Orth_{2,5678}(2U(4)\oplus A_1)$ be the group generated by $\Orth_{2}(2U(4)\oplus A_1)$ and reflections associated to the divisor of $\Theta_5\Theta_7$. Then this group also contains the reflections associated to the divisor of $\Theta_6\Theta_8$. Using Lemma \ref{lem:2U4A1jacobian} we derive that the Jacobian of $\Theta_1^2$, $\Theta_2^2$, $\Theta_3^2\Theta_4^2$, $\Theta_5^2\Theta_6^2$ equals $\Phi_{5,A_1}\prod_{i=1}^8\Theta_i$ up to a non-zero constant multiple, and have proved the following:
$$ 
M_*(\Orth_{2,5678}(2U(4)\oplus A_1))=\CC[\Theta_1^2,\Theta_2^2, \Theta_3^2\Theta_4^2, \Theta_5^2\Theta_6^2].
$$

Altogether we have computed the algebras of modular forms for eight subgroups of $\Orth^+(2U\oplus A_1)$:
\begin{align*}
&\Orth_{1}'(2U(4)\oplus A_1) \subsetneq \Orth_{0}(2U(4)\oplus A_1) \subsetneq \Orth_{1}(2U(4)\oplus A_1) \subsetneq \Orth_{1,1}(2U(4)\oplus A_1) \subsetneq \\ 
&\subsetneq \Orth_{1,12}(2U(4)\oplus A_1) \subsetneq \Orth_{2}(2U(4)\oplus A_1) \subsetneq \Orth_{2,56}(2U(4)\oplus A_1) \subsetneq \Orth_{2,5678}(2U(4)\oplus A_1).
\end{align*}

\begin{remark}
There is no multiplier system $v$ of $\Orth_2(U(4)\oplus A_1)$ which equals 1 on all reflections. If such a $v$ existed, then $v_4/v$ would define a character of $\Orth_2(U(4)\oplus A_1)$ which equals 1 on all reflections associated to the divisor of $\Theta_1\Theta_2\Theta_3$ but is $-1$ on all reflections associated to the divisor of $\Theta_4$, which contradicts \eqref{eq:O2}. 
\end{remark}

\begin{remark}
The Baily-Borel compactification of the modular variety $\cD_{3}/ \Orth_1(2U(4)\oplus A_1)$ is a projective space, but the algebra $M_*(\Orth_1(2U(4)\oplus A_1))$ of modular forms of integral weight is not free. It is well-known that if $M_*(\Gamma)$ is free then $(\cD / \Gamma)^*$ is a weighted projective space (see \cite{BB66}), but the above example shows that the converse does not hold. Eberhard Freitag suggested to the authors that the following statement may hold. 
\begin{conjecture}
If the modular variety $(\cD / \Gamma)^*$ is a weighted projective space, then there exists a weight $k_0$ and a multiplier system $v_0$ of weight $k_0$ such that the graded algebra 
$$
M_*(\Gamma, (k_0,v_0)):=\bigoplus_{k\in \NN} M_{kk_0}(\Gamma, v_0^k)
$$
is freely generated.
\end{conjecture}
\end{remark}

\begin{remark}
There are twenty linearly independent modular forms of weight $1/2$ for the Weil representation attached to $2U(4)\oplus A_1$, and their images under the additive theta lift (\cite{Bor98}, Theorem 14.3) span the 10-dimensional space of modular forms of weight one for $\Orth_1(2U(4)\oplus A_1)$. 
\end{remark}

\section{The \texorpdfstring{$2U(2)\oplus 2A_1$}{} lattice}
In this section we prove Matsumoto's theorem in the context of modular forms on $2U(2)\oplus 2A_1$. We first work out some Borcherds products on $2U(2)\oplus 2A_1$.

\begin{lemma}\label{lem:2A1inputs}
\noindent
\begin{enumerate}
\item There are $10$ holomorphic Borcherds products of singular weight 1. Their inputs have principal parts of the form
$$
q^{-1/4} (e_{v_1} + e_{v_2}) + 2e_0, \quad (v_1,v_1)=(v_2,v_2)=1/2, \quad \ord(v_1)=\ord(v_2)=2,
$$
where $v_1$ and $v_2$ are images of each other under swapping the two $A_1$ components. In particular they are 2-reflective modular forms.  We label the ten forms $F_i$, $1\leq i \leq 10$ such that with respect to the Gram matrix
$$
\begin{psmallmatrix}
0 & 0 & 0 & 0 & 0 & 2 \\ 
0 & 0 & 0 & 0 & 2 & 0 \\ 
0 & 0 & 2 & 0 & 0 & 0 \\ 
0 & 0 & 0 & 2 & 0 & 0 \\ 
0 & 2 & 0 & 0 & 0 & 0 \\ 
2 & 0 & 0 & 0 & 0 & 0
\end{psmallmatrix}
$$ their principal parts are as follows:

\begin{small}
\begin{align*}
F_1: \; & q^{-1/4}(e_{(0, 1/2, 1/2, 0, 0, 0)} + e_{(0, 1/2, 0, 1/2, 0, 0)}); \; &F_2: \;  &q^{-1/4}(e_{(1/2, 0, 1/2, 0, 0, 0)} + e_{(1/2, 0, 0, 1/2, 0, 0)});\\
F_3: \; & q^{-1/4}(e_{(1/2, 1/2, 1/2, 0, 1/2, 1/2)} + e_{(1/2, 1/2, 0, 1/2, 1/2, 1/2)}); \; &F_4: \;  &q^{-1/4}(e_{(1/2, 0, 1/2, 0, 1/2, 0)} + e_{(1/2, 0, 0, 1/2, 1/2, 0)});\\
F_5: \; & q^{-1/4}(e_{(0, 1/2, 1/2, 0, 0, 1/2)} + e_{(0, 1/2, 0, 1/2, 0, 1/2)}); \; &F_6: \;  &q^{-1/4}(e_{(1/2, 1/2, 1/2, 0, 0, 0)} + e_{(1/2, 1/2, 0, 1/2, 0, 0)});\\
F_7: \; & q^{-1/4}(e_{(0, 0, 1/2, 0, 1/2, 1/2)} + e_{(0, 0, 0, 1/2, 1/2, 1/2)}); \; &F_8: \;  &q^{-1/4}(e_{(0, 0, 1/2, 0, 1/2, 0)} + e_{(0, 0, 0, 1/2, 1/2, 0)});\\
F_9: \; & q^{-1/4}(e_{(0, 0, 1/2, 0, 0, 1/2)} + e_{(0, 0, 0, 1/2, 0, 1/2)}); \; &F_{10}:\;  &q^{-1/4}(e_{(0, 0, 1/2, 0, 0, 0)} + e_{(0, 0, 0, 1/2, 0, 0)}).
\end{align*}
\end{small}
The ten $F_i$ are linearly independent over $\CC$. We define $\Phi_{10,2A_1}:=\prod_{j=1}^{10}F_j$.

\item There are $15$ reflective Borcherds products of weight $2$ with principal parts of the form $q^{-1/2}e_u$, where $(u,u)=1$, $\ord(u)=2$ and $u$ is invariant under the swapping of two $A_1$ components. 

\item There is a holomorphic Borcherds product of weight 4, which we denote $\Phi_{4, 2A_1}$, whose input has the principal part $q^{-1/2} e_{(0, 0, 1/2, 1/2, 0, 0)}$. The reflection $\sigma_1$ associated to the vector $(0,0,1/2,1/2,0,0)$ swaps the two $A_1$ components.
\end{enumerate}
\end{lemma}

Let $\Orth_1(2U(2)\oplus 2A_1)$ be the subgroup of $\Orth^+(2U(2)\oplus 2A_1)$ generated by all reflections associated to the divisors of $\Phi_{10,2A_1}$ and $\Phi_{4,2A_1}$. The inputs of all $F_i$ and of $\Phi_{4, 2A_1}$ are invariant under the reflection $\sigma_1$, so $F_i$ and $\Phi_{4, 2A_1}$ are modular forms for $\Orth_1(2U(2)\oplus 2A_1)$. Each $F_i$ has a quadratic character on $\Orth_1(2U(2) \oplus 2A_1)$ so their squares $F_i^2$ have trivial character. 

\begin{lemma}\label{lem:2A1}
The Jacobian of $F_i^2$, $1\leq i \leq 5$ equals $\Phi_{10,2A_1}\Phi_{4,2A_1}$ up to a non-zero constant multiple. In particular, the forms $F_i$ for $1\leq i \leq 5$ are algebraically independent over $\CC$.
\end{lemma}
\begin{proof}
By computing the first few terms of the Fourier expansion we find that $J = J(F_i^2, 1 \le i \le 5)$ is not identically zero. By Theorem \ref{th:Jacobian}, $J$ is a modular form of weight $14$ and vanishes on mirrors of all reflections in $\Orth_1(2U(2)\oplus 2A_1)$. In particular $J/(\Phi_{10,2A_1}\Phi_{4,2A_1})$ is a holomorphic modular form of weight $0$ and therefore constant.
\end{proof}
By Theorem \ref{th:converseJacobian}, we obtain the following result which is equivalent to Matsumoto's theorem.
\begin{theorem}\label{th:Matsumoto}
\begin{align*}
M_*(\Orth_1(2U(2)\oplus 2A_1))&=\CC[F_i^2, 1\leq i \leq 5],\\
(\cD_{4}/ \Orth_1(2U(2)\oplus 2A_1))^* &\cong \PP^4(\CC).
\end{align*}
\end{theorem}

\begin{corollary}\label{cor:Matsumoto}
The squares $F_i^2$ span a five-dimensional space over $\CC$ and they satisfy the four-term relations \begin{align*} F_1^2 - F_5^2 + F_6^2 - F_{10}^2 &= 0, \\ F_2^2 - F_5^2 - F_8^2 - F_9^2 &= 0, \\ F_3^2 - F_5^2 + F_6^2 - F_9^2 &= 0, \\ F_4^2 - F_6^2 - F_8^2 + F_{10}^2 &= 0, \\ F_7^2 - F_8^2 - F_9^2 + F_{10}^2 &= 0.\end{align*}
The vector space spanned by the $15$ weight two products in Lemma \ref{lem:2A1inputs} and the ten squares $F_i^2$ has dimension $5$.
\end{corollary}
\begin{proof} The exact form of the relations among the $F_i^2$ can be read off of their Fourier expansions. The 15 weight two products have trivial character on $\Orth_1(2U(2) \oplus 2A_1)$ and therefore lie in the span of $F_1^2,...,F_5^2$.
\end{proof}

We will now obtain an analogue of Igusa's theorem for Hermitian modular forms of degree 2 over the Gaussian numbers. The Jacobian $J=J(F_i^2, 1\leq i\leq 5)$ is the product of the ten $F_i$ and $\Phi_{4,2A_1}$. 
To apply Theorem \ref{th:freeJacobian} we have to show that $\Phi_{4,2A_1}$ is irreducible as a modular form on $\Orth_1(2U(2)\oplus 2A_1)$, i.e. it is not a product of two non-constant modular forms with characters (or multiplier systems) on $\Orth_1(2U(2)\oplus 2A_1)$. Since $\Phi_{4,2A_1}^2\in M_8(\Orth_1(2U(2)\oplus 2A_1))$, it can be expressed as a polynomial $P$ in terms of the five $F_i^2$.  We computed this polynomial and found that it is irreducible. From Theorem \ref{th:freeJacobian},  it follows that $\Phi_{4,2A_1}$ is irreducible on $\Orth_1(2U(2)\oplus 2A_1)$.

By Theorem \ref{th:freeJacobian}, there are exactly $2048$ characters of $\Orth_1(2U(2) \oplus 2A_1)$ and they are generated by the basic characters $\chi_i$ of $F_i$ and the character of $\Phi_{4,2A_1}$. Let $\Orth_1'(2U(2)\oplus 2A_1)$ be the commutator subgroup of $\Orth_1(2U(2)\oplus 2A_1)$. Then
$$
M_*(\Orth_1'(2U(2)\oplus 2A_1))=\CC[\Phi_{4,2A_1}, F_i, 1\leq i \leq 10].
$$

Let $\Orth_0$ be the subgroup
$$
\Orth_0(2U(2)\oplus 2A_1)=\{ \gamma\in \Orth_1(2U(2)\oplus 2A_1): \chi_i(\gamma)=1, \, 1\leq i \leq 10  \}.
$$
We have the following result immediately.
\begin{theorem}\label{th:analogue-Igusa}
$$
M_*(\Orth_0(2U(2)\oplus 2A_1))=\CC[F_i, 1\leq i\leq 10 ].
$$
\end{theorem}

Similarly, we define another subgroup of $\Orth_1(2U(2)\oplus 2A_1)$ via
$$
\Orth_{1'}(2U(2)\oplus 2A_1)=\{ \gamma\in \Orth_1(2U(2)\oplus 2A_1): \chi_i(\gamma)=1, \; 1\leq i \leq 5  \}.
$$
The group $\Orth_{1'}(2U(2)\oplus 2A_1)$ contains all reflections associated to the divisor of $\Phi_{4,2A_1}\prod_{j=6}^{10}F_j$. The forms $F_i$, $1\leq i\leq 5$, are modular with trivial character on $\Orth_{1'}(2U(2)\oplus 2A_1)$ and their Jacobian equals $\Phi_{4,2A_1}\prod_{j=6}^{10}F_j$ up to a non-zero constant multiple by Lemma \ref{lem:2A1}. Therefore we obtain the following result.
\begin{theorem}\label{th:projective4}
\begin{align*}
M_*(\Orth_{1'}(2U(2)\oplus 2A_1))&=\CC[F_i, 1\leq i \leq 5],\\
(\cD_{4}/ \Orth_{1'}(2U(2)\oplus 2A_1))^* &\cong \PP^4(\CC).
\end{align*}
\end{theorem}
The above theorem also implies that $\Orth_{1'}(2U(2)\oplus 2A_1)$ is generated by all reflections associated to the divisor of $\Phi_{4,2A_1}\prod_{j=6}^{10}F_j$. (Note that the same argument applies with $F_1,...,F_5$ replaced by any of the 162 linearly independent sets of five squares of products $F_i^2$. These sets can be read off of the four-term relations of Corollary \ref{cor:Matsumoto}.)

We define a larger subgroup of $\Orth_1(2U(2)\oplus 2A_1)$ via
$$
\Orth_{1''}(2U(2)\oplus 2A_1)=\{ \gamma\in \Orth_1(2U(2)\oplus 2A_1): \chi_1(\gamma)=1\},
$$
i.e. the subgroup generated by all reflections associated to the divisor of $\Phi_{4,2A_1}\prod_{j=2}^{10}F_j$.
A similar argument to the above theorem yields the following result.
\begin{theorem}\label{th:projective5}
\begin{align*}
M_*(\Orth_{1''}(2U(2)\oplus 2A_1))&=\CC[F_1, F_2^2, F_3^2, F_4^2, F_5^2],\\
(\cD_{4}/ \Orth_{1''}(2U(2)\oplus 2A_1))^* &\cong\PP(1,2,2,2,2) \cong \PP^4(\CC).
\end{align*}
\end{theorem}
The ring of even-weight modular forms for $\Orth_{1''}(2U(2)\oplus 2A_1)$ is freely generated in weight two (and indeed is exactly the ring from Theorem \ref{th:Matsumoto}):
$$
M_{2*}(\Orth_{1''}(2U(2)\oplus 2A_1))^{\text{even}} = \CC[F_1^2,F_2^2,F_3^2,F_4^2,F_5^2].
$$

\begin{remark}\label{rem:irreducible-Phi4}
There is a six-dimensional space of modular forms of weight $1$ for the Weil representation attached to $2U(2)\oplus 2A_1$, and these map to a five-dimensional space of modular forms of weight $2$ with trivial character on the discriminant kernel $\widetilde{\Orth}^+(2U(2)\oplus 2A_1)$ under the additive theta lift. Applying Theorem \ref{th:converseJacobian} to $\widetilde{\Orth}^+(2U(2)\oplus 2A_1)$, we find that $M_*(\widetilde{\Orth}^+(2U(2)\oplus 2A_1))$ is freely generated by the five additive lifts. Moreover, $\widetilde{\Orth}^+(2U(2)\oplus 2A_1)$ is generated by all 2-reflections and then  
$$
\widetilde{\Orth}^+(2U(2)\oplus 2A_1)=\Orth_1(2U(2)\oplus 2A_1).
$$
By the Eichler criterion (cf. Proposition 3.3 of \cite{GHS2009}), the divisor of $\Phi_{4,2A_1}$ is irreducible on $\widetilde{\Orth}^+(2U(2)\oplus 2A_1)$.  In this way we obtain a different proof that $\Phi_{4,2A_1}$ is irreducible on $\Orth_1(2U(2)\oplus 2A_1)$. We also see that the square of each of the singular weight products is an additive lift.
\end{remark}

\begin{remark}
By \cite[Lemma 6.1]{GN18}, we have
$$
\Orth^+(2U(2)\oplus 2A_1)\cong \Orth^+(2U\oplus 2A_1).
$$
Thus $\Phi_{10,2A_1}$ can be regarded as a modular form on $2U\oplus 2A_1$. In the interpretation as Hermitian modular forms this is the Borcherds product $\phi_{10}$ in \cite[Corollary 4]{DK03}, which was realized earlier by Freitag \cite{Fre67} as the product of ten theta constants. Indeed, our $F_i$ are exactly the ten theta constants when interpreted as Hermitian modular forms.
\end{remark}

\section{The \texorpdfstring{$2U(3)\oplus A_2$}{} lattice}
In this section we reprove the theorem of Freitag and Salvati Manni in the context of modular forms on $2U(3)\oplus A_2$. We first work out some Borcherds products on $2U(3)\oplus A_2$:
\vspace{2mm}
\begin{enumerate}
\item There are 45 holomorphic Borcherds products of singular weight 1. Their inputs have principal parts of the form
$$
q^{-1/3} (e_v+e_{-v}) + 2e_0, \quad (v,v)=2/3, \; \ord(v)=3.
$$
The product of these 45 forms is a reflective modular form $\Phi_{45,A_2}$ of weight 45 which can be viewed as a 2-reflective modular form for $\Orth^+(2U\oplus A_2)$. 
\item There is a holomorphic Borcherds product of weight 9 whose input has principal part
$$
(q^{-1}+18)e_0.
$$
We label this form $\Phi_{9,A_2}$. It is a 2-reflective modular form on $\Orth^+(2U(3)\oplus A_2)$ and can be regarded as a reflective modular form for $\Orth^+(2U\oplus A_2)$. 
\end{enumerate}
\vspace{2mm}

Let $\Orth_1(2U(3)\oplus A_2)$ be the subgroup of $\Orth^+(2U(3)\oplus A_2)$ generated by all 2-reflections, i.e. reflections associated to the divisor of $\Phi_{9,A_2}$. It is clear that $\Orth_1(2U(3)\oplus A_2)$ is a subgroup of $\widetilde{\Orth}^+(2U(3)\oplus A_2)$. By considering their divisors, we see that the 45 weight $1$ products have trivial character on $\Orth_1(2U(3)\oplus A_2)$.

It will again be convenient to use the notion of $*$-sets.
\begin{definition}
A $*$-set (of type $2U(3)\oplus A_2$) is a set of five products of weight $1$ on $2U(3)\oplus A_2$ whose inputs are invariant under all reflections associated to the divisors of any of the five products.
\end{definition}

$*$-sets of type $2U(3) \oplus A_2$ satisfy properties analogous to Lemma \ref{lem:*-sets}:

\begin{lemma}
\noindent
\begin{enumerate}
\item There are exactly $27$ $*$-sets.
\item Every product of weight $1$ belongs to exactly three $*$-sets.
\item There are $720$ pairs of weight $1$ products that do not belong to a $*$-set. The remaining $270$ pairs belong to a unique $*$-set.
\item The five elements of any $*$-set are linearly independent over $\CC$.
\end{enumerate}
\end{lemma}

\begin{lemma}
There is a $*$-set $\{G_1, G_2,G_3,G_4,G_5\}$ whose Jacobian equals $\Phi_{9,A_2}$ up to a non-zero constant multiple.
\end{lemma}
\begin{proof} In the notation of Appendix B, we used the forms $G_1, G_{14}, G_{15}, G_{23}, G_{27}$ which form a $*$-set. We computed their Jacobian to precision $10$ and found that it is not identically zero (and indeed agrees with the Fourier expansion of $\Phi_{9, A_2}$). By the same argument as Lemma \ref{lem:2U4A1jacobian}, the quotient $J / \Phi_{9, A_2}$ is holomorphic of weight zero and therefore constant.
\end{proof}
By applying Theorem \ref{th:converseJacobian} we obtain the theorem of Freitag and Salvati Manni in the context of orthogonal groups.
\begin{theorem}\label{th:FSM} The $45$ singular-weight products on $2U(3) \oplus A_2$ span a five-dimensional space over $\CC$. Any five linearly independent products $G_1,...,G_5$ are algebraically independent and generate the algebra of modular forms:
\begin{align*}
M_*(\Orth_1(2U(3)\oplus A_2)) &=\CC[G_1,G_2,G_3,G_4,G_5],\\
(\cD_{4}/ \Orth_1(2U(3)\oplus A_2))^* &\cong \PP^4(\CC).
\end{align*}
\end{theorem}

We also have the following analogue of Perna's theorem.
\begin{theorem}\label{th:FSMPerna}
Let $\{G_1,...,G_5\}$ be a $*$-set and let $\Orth_2(2U(3)\oplus A_2)$ be the subgroup generated by $\Orth_1(2U(3)\oplus A_2)$ and the reflections associated to the divisor of $\prod_{j=1}^5G_j$. Then
\begin{align*}
M_*(\Orth_2(2U(3)\oplus A_2)) &=\CC[G_1^2,G_2^2,G_3^2,G_4^2,G_5^2],\\
(\cD_{4}/ \Orth_2(2U(3)\oplus A_2))^* &\cong \PP^4(\CC).
\end{align*}
\end{theorem}

In Remark \ref{rem:irreducible-Phi9} below, we will see that the squared Jacobian $\Phi_{9, A_2}^2$ is irreducible in $M_*(\Orth_1(2U(3) \oplus A_2))$. From Theorem \ref{th:freeJacobian} it follows that $\det$ is the only non-trivial character of $\Orth_1(2U(3)\oplus A_2)$. Let $\Orth_1'(2U(3)\oplus A_2)$ be the commutator subgroup of $\Orth_1(2U(3)\oplus A_2)$. Then
$$
M_*(\Orth_1'(2U(3)\oplus A_2))=\CC[\Phi_{9,A_2}, G_i, 1\leq i \leq 5].
$$

As in \S \ref{sec:A1} we can construct a tower of free algebras of modular forms. We fix a $*$-set $\{G_1,...,G_5\}$. Let $\Orth_{1,1}(2U(3)\oplus A_2)$, $\Orth_{1,12}(2U(3)\oplus A_2)$, $\Orth_{1,123}(2U(3)\oplus A_2)$ and $\Orth_{1,1234}(2U(3)\oplus A_2)$ be the subgroups generated by reflections associated to the divisors of $\Phi_{9,A_2}G_1$, $\Phi_{9,A_2}\prod_{j=1}^2G_j$, $\Phi_{9,A_2}\prod_{j=1}^3G_j$ and $\Phi_{9,A_2}\prod_{j=1}^4G_j$ respectively. It is easy to derive the following structure results:
\begin{align*}
M_*(\Orth_{1,1}(2U(3)\oplus A_2))=\CC[G_1^2,G_2,G_3,G_4,G_5],\\
M_*(\Orth_{1,12}(2U(3)\oplus A_2))=\CC[G_1^2,G_2^2,G_3,G_4,G_5],\\
M_*(\Orth_{1,123}(2U(3)\oplus A_2))=\CC[G_1^2,G_2^2,G_3^2,G_4,G_5],\\
M_*(\Orth_{1,1234}(2U(3)\oplus A_2))=\CC[G_1^2,G_2^2,G_3^2,G_4^2,G_5].
\end{align*}
From the last of these we obtain another realization of $\PP^4(\CC)$ as an orthogonal modular variety:
\begin{theorem}\label{th:projective8}
$$
(\cD_{4}/\Orth_{1,1234}(2U(3)\oplus A_2))^*\cong \PP(1,2,2,2,2) \cong \PP^4(\CC).
$$
\end{theorem}

\begin{remark}\label{rem:irreducible-Phi9}
The space of invariants of the Weil representation attached to $2U(3) \oplus A_2$ is ten-dimensional, and these are mapped to a five-dimensional space of modular forms of weight one with trivial character on the full discriminant kernel $\widetilde{\Orth}^+(2U(3)\oplus A_2)$ under the additive theta lift. In particular the 45 singular weight products are all additive lifts. Using the argument in Remark \ref{rem:irreducible-Phi4} we conclude that $M_*(\widetilde{\Orth}^+(2U(3)\oplus A_2))$ is freely generated by the five additive lifts and that
$$
\widetilde{\Orth}^+(2U(3)\oplus A_2)=\Orth_1(2U(3)\oplus A_2),
$$
and the Eichler criterion implies that $\Phi_{9,A_2}$ is irreducible in $M_*(\Orth_1(2U(3)\oplus A_2))$.
\end{remark}

\section{The \texorpdfstring{$2U(2)\oplus A_1$}{} lattice}
In this section we will find three interesting free algebras of modular forms associated to the lattice $2U(2)\oplus A_1$. We first describe some Borcherds products on $2U(2)\oplus A_1$ of small weight.
\vspace{2mm}
\begin{enumerate}
\item There are nine holomorphic Borcherds products of weight $1$ and one holomorphic product of weight $2$ on $2U(2)\oplus A_1$. All have principal parts of the form
$$
q^{-1/4}e_v, \quad (v,v)=1/2, \; \ord(v)=2.
$$
The product of these ten forms, which we denote $\Phi_{11,A_1(2)}$, is a 2-reflective modular form of weight $11$ on $2U(2)\oplus A_1$. This can also be viewed as a modular form on $2U\oplus A_1(2)$ by \citep[Lemma 6.1]{GN18} because
$$
\Orth^+(2U(2)\oplus A_1)\cong \Orth^+(2U\oplus A_1(2)).
$$
\item There are six holomorphic Borcherds products $f_1,...,f_6$ of weight $2$ with principal parts
$$
q^{-1/2}e_v, \quad (v,v)=1, \; \ord(v)=2.
$$
\end{enumerate}
\vspace{2mm}

We fix the Gram matrix $\begin{psmallmatrix} 0 & 0 & 0 & 0 & 2 \\ 
0 & 0 & 0 & 2 & 0 \\ 
0 & 0 & 2 & 0 & 0 \\ 
0 & 2 & 0 & 0 & 0 \\ 
2 & 0 & 0 & 0 & 0 \end{psmallmatrix}$ and label the six products $f_1,...,f_6$ above such that their principal parts are as follows:
\begin{align*} 
f_1 : \; &v = (0, 1/2, 0, 1/2, 0);
&f_2 : \; &v = (0, 1/2, 0, 1/2, 1/2); \\
f_3 : \; &v = (1/2, 0, 0, 1/2, 1/2);
&f_4 : \; &v = (1/2, 0, 0, 0, 1/2); \\
f_5 : \; &v = (1/2, 1/2, 0, 0, 1/2);
&f_6 : \; &v = (1/2, 1/2, 0, 1/2, 0).
\end{align*}

Let $\Orth_1(2U(2)\oplus A_1)$ be the subgroup of $\Orth^+(2U(2)\oplus A_1)$ generated by all reflections associated to the divisor of $\Phi_{11,A_1(2)}$. 

\begin{lemma}
\noindent
\begin{enumerate}
\item The six products $f_i$ are modular forms with trivial character on $\Orth_1(2U(2)\oplus A_1)$.
\item The Jacobian $J(f_1,f_2,f_3,f_4)$ equals $\Phi_{11,A_1(2)}$ up to a non-zero multiple.
\end{enumerate}
\end{lemma}
\begin{proof} (1) This can be seen from the divisor of the $f_i$. \\
(2) We checked by computer that $J = J(f_1,...,f_4)$ is nonzero (and indeed equals $768\Phi_{11,A_1(2)}$ up to precision $O(q, s)^{10}$). Applying Theorem \ref{th:Jacobian} (4) as in the earlier sections shows that $J / \Phi_{11, A_1(2)}$ is holomorphic of weight zero and therefore a constant.
\end{proof}

Therefore we can apply Theorem \ref{th:converseJacobian} to this situation. This yields another realization of $\PP^3(\CC)$ as a modular variety.
\begin{theorem}\label{th:projective9} The six products $f_1,...,f_6$ span a four-dimensional space and satisfy the relations $$f_1 + f_3 + f_5 = f_2 + f_4 + f_6 = 0.$$ Any four that are linearly independent are algebraically independent and generate the ring of modular forms for $\Orth_1(2U(2) \oplus A_1)$:
\begin{align*}
M_*(\Orth_1(2U(2)\oplus A_1)) &=\CC[f_1,f_2,f_3,f_4],\\
(\cD_{3}/ \Orth_1(2U(2)\oplus A_1))^* &\cong \PP^3(\CC).
\end{align*}
\end{theorem}
\begin{proof} Theorem \ref{th:converseJacobian} implies everything except the exact form of the relations among the $f_i$, which can be determined from their Fourier expansions.
\end{proof}

\begin{remark} The squares of the nine Borcherds products $b_i$ of weight one are modular forms without character for $\Orth_1(2U(2) \oplus A_1)$ and therefore lie in the span of $f_1,...,f_4$. Indeed they also span this space and they satisfy five three-term linear relations of the form $b_i^2 + b_j^2 = b_k^2$ as one can check from their Fourier expansions.
\end{remark}

Choose any linearly independent squares of weight one Borcherds products $b_1^2$, $b_2^2$, $b_3^2$ and $b_4^2$ as in the remark. Let $\Orth_{1'}(2U(2)\oplus A_1)$ be the subgroup of $\Orth_1(2U(2)\oplus A_1)$ generated by all reflections associated to the divisor of $\Phi_{11,A_1(2)}/ (\prod_{j=1}^4b_j)$. Similarly to the case of $2U(2)\oplus 2A_1$,  the four forms $b_j$ are modular with trivial character on $\Orth_{1'}(2U(2)\oplus A_1)$ and their Jacobian equals $\Phi_{11,A_1(2)}/ (\prod_{j=1}^4b_j)$ up to a nonzero multiple. From this we obtain another realization of $\PP^3(\CC)$ as a modular variety:
\begin{theorem}\label{th:projective10}
\begin{align*}
M_*(\Orth_{1'}(2U(2)\oplus A_1)) &=\CC[b_1,b_2,b_3,b_4],\\
(\cD_{3}/ \Orth_{1'}(2U(2)\oplus A_1))^* &\cong \PP^3(\CC).
\end{align*}
\end{theorem}

Furthermore, we define $\Orth_{1''}(2U(2)\oplus A_1)$ as the subgroup of $\Orth_1(2U(2)\oplus A_1)$ generated by all reflections associated to the divisor of $\Phi_{11,A_1(2)}/ b_1$. Then
\begin{theorem}\label{th:projective11}
\begin{align*}
M_*(\Orth_{1''}(2U(2)\oplus A_1)) &=\CC[b_1,b_2^2,b_3^2,b_4^2],\\
(\cD_{3}/ \Orth_{1''}(2U(2)\oplus A_1))^* &\cong \PP(1,2,2,2)\cong \PP^3(\CC).
\end{align*}
\end{theorem}

\begin{remark}
The space of modular forms of weight $3/2$ for the Weil representation attached to $2U(2)\oplus A_1$ is five-dimensional, and these forms map to a four-dimensional space of modular forms of weight $2$ under the additive theta lift. In particular, every modular form in $M_2(\Orth_1(2U(2)\oplus A_1))$ is an additive lift and therefore has trivial character on the full discriminant kernel. Similarly to Remark \ref{rem:irreducible-Phi4}, we find that $M_*(\widetilde{\Orth}^+(2U(2)\oplus A_1))$ is freely generated by the four additive lifts and that
$$
\widetilde{\Orth}^+(2U(2)\oplus A_1)=\Orth_1(2U(2)\oplus A_1).
$$
\end{remark}

Similarly to the previous sections, the decomposition of the Jacobian $\Phi_{11, A_1(2)}$ determines the structure of the algebra of modular forms for the commutator group $\Orth_1'(2U(2) \oplus A_1)$ and some related groups. We omit the details.

\section{The \texorpdfstring{$2U(3)\oplus A_1$}{} lattice}
In this section we determine an interesting free algebra of modular forms on $2U(3)\oplus A_1$. We first work out Borcherds products on $2U(3)\oplus A_1$:
\vspace{2mm}
\begin{enumerate}
\item There are 29 holomorphic Borcherds products of weight $1$ on $2U(3)\oplus A_1$. Sixteen of them have principal parts of the form
$$
q^{-1/4}(e_v+e_{-v}), \quad (v,v)=1/2, \; \ord(v)=6,
$$
and they are not reflective modular forms. Twelve of them have principal parts
$$
q^{-1/3}(e_v+e_{-v}), \quad (v,v)=2/3, \; \ord(v)=3,
$$
and they are reflective modular forms. The last one is a 2-reflective modular form denoted by $\Delta_1$ with principal part
$$
q^{-1/4}e_{(0, 0, 1/2, 0, 0)}
$$
with respect to the Gram matrix
$$
\begin{psmallmatrix}
0 & 0 & 0 & 0 & 3 \\ 
0 & 0 & 0 & 3 & 0 \\ 
0 & 0 & 2 & 0 & 0 \\ 
0 & 3 & 0 & 0 & 0 \\ 
3 & 0 & 0 & 0 & 0
\end{psmallmatrix}.
$$
\item There is a holomorphic Borcherds product of weight $7$ with principal part $q^{-1}e_0$. We label this 2-reflective form $\Phi_{7,A_1(3)}$.
\end{enumerate}
\vspace{2mm}

Let $\Orth_1(2U(3)\oplus A_1)$ be the subgroup of $\Orth^+(2U(3)\oplus A_1)$ generated by all reflections associated to the divisor of $\Phi_{7,A_1(3)}$. The 28 products of weight $1$ other than $\Delta_1$ are modular forms of trivial character on $\Orth_1(2U(3)\oplus A_1)$. 

\begin{lemma}
There are four products $g_1,...,g_4$ of weight $1$ whose Jacobian $J(g_1,g_2,g_3,g_4)$ equals $\Phi_{7,A_1(3)}$ up to a non-zero multiple.
\end{lemma}
\begin{proof} We used the products $g_i$ whose inputs have principal parts $q^{-1/3} (e_{v_i} + e_{-v_i})$, where \begin{align*} &v_1=(1/3, 2/3, 0, 2/3, 0), \quad
& v_2&=(1/3, 2/3, 0, 1/3, 2/3), \\
&v_3=(1/3, 1/3, 0, 2/3, 2/3), \quad  &v_4&=(1/3, 1/3, 0, 1/3, 0), \end{align*} (although it will turn out that any four linearly independent forms not including $\Delta_1$ will do) and using Fourier series computed that their Jacobian $J$ is nonzero (and equals $(12 \zeta_3 + 6) \Phi_{7, A_1(3)}$ up to precision $O(q,s)^{10}$). As in the previous sections $J / \Phi_{7, A_1(3)}$ is holomorphic of weight zero and therefore constant.
\end{proof}

By applying Theorem \ref{th:converseJacobian} we obtain the following theorem.
\begin{theorem}\label{th:projective13} Every linearly independent set $g_1,...,g_4$ of weight one products that does not include $\Delta_1$ is algebraically independent and generates the ring of modular forms for $\Orth_1(2U(3) \oplus A_1)$:
\begin{align*}
M_*(\Orth_1(2U(3)\oplus A_1)) &=\CC[g_1,g_2,g_3,g_4],\\
(\cD_{3}/ \Orth_1(2U(3)\oplus A_1))^* &\cong \PP^3(\CC).
\end{align*}
\end{theorem}

In particular, the $28$ products of weight $1$ other than $\Delta_1$ span a $4$-dimensional space. ($\Delta_1$ has a nontrivial character on $\Orth_1(2U(3) \oplus A_1)$ and in particular does not lie in their span.)

\begin{remark}
The space of modular forms of weight $1/2$ for the Weil representation attached to $2U(3)\oplus A_1$ is $8$-dimensional, and these map to a four-dimensional space of modular forms of weight one under the additive theta lift. It follows that all of the $28$ products of weight $1$ other than $\Delta_1$ are additive lifts.  Similarly to $2U(2)\oplus 2A_1$, we conclude that $M_*(\widetilde{\Orth}^+(2U(3)\oplus A_1))$ is freely generated by the four additive lifts and that
$$
\widetilde{\Orth}^+(2U(3)\oplus A_1)=\Orth_1(2U(3)\oplus A_1).
$$
By the Eichler criterion, the forms $\Delta_1$ and $\Phi_{7,A_1(3)}/\Delta_1$ are irreducible on $\Orth_1(2U(3)\oplus A_1)$.
\end{remark}

\section{The \texorpdfstring{$U(4)\oplus U(2)\oplus A_1$}{} lattice}
In this section we determine two interesting free algebras of modular forms on $U(4)\oplus U(2)\oplus A_1$. We first work out Borcherds products on $U(4)\oplus U(2)\oplus A_1$:
\vspace{2mm}
\begin{enumerate}
\item There are 8 holomorphic Borcherds products of weight $1/2$ and 3 holomorphic products of weight $1$ on $U(4)\oplus U(2)\oplus A_1$, all with principal parts of the form
$$
q^{-1/4}e_v, \quad (v,v)=1/2, \; \ord(v)=2.
$$
Their product is a 2-reflective modular form of weight $7$ which we label $\Phi_{7,A_1(4)}$.
\item There are 16 other holomorphic Borcherds products of weight $1$. They are all reflective. Twelve of them have input forms with principal parts
$$
q^{-1/4}(e_v+e_{-v}), \quad (v,v)=1/2, \; \ord(v)=4.
$$
The remaining four have input forms with principal parts 
$$
q^{-1/2}e_v, \quad (v,v)=1, \; \ord(v)=2.
$$
\end{enumerate}

Let $\Orth_1(U(4)\oplus U(2)\oplus A_1)$ be the subgroup of $\Orth^+(U(4)\oplus U(2)\oplus A_1)$ generated by all reflections associated to the divisor of $\Phi_{7,A_1(4)}$. Then the 16 products of weight $1$ in (2) are modular forms of trivial character on $\Orth_1(U(4)\oplus U(2)\oplus A_1)$.

Let $h_1, ..., h_4$ denote the four products in (2) whose input forms have principal part $q^{-1/2} e_v$ with $v$ of order two.
\begin{lemma}
The Jacobian $J = J(h_1,h_2,h_3,h_4)$ is equals $\Phi_{7,A_1(4)}$ up to a non-zero constant multiple. The four forms $h_1$, $h_2$, $h_3$ and $h_4$ are algebraically independent over $\CC$.
\end{lemma}
\begin{proof} As in the previous sections the quotient $J / \Phi_{7, A_1(4)}$ is holomorphic of weight zero and therefore constant. We checked by computer that this constant is not zero.
\end{proof}
Theorem \ref{th:converseJacobian} yields the following structure theorem:
\begin{theorem}\label{th:projective14}
\begin{align*}
M_*(\Orth_1(U(4)\oplus U(2)\oplus A_1)) &=\CC[h_1,h_2,h_3,h_4],\\
(\cD_{3}/ \Orth_1(U(4)\oplus U(2)\oplus A_1))^* &\cong \PP^3(\CC).
\end{align*}
\end{theorem}

\begin{corollary}
The $16$ products of weight one in $(2)$ span a $4$-dimensional space.  The $8$ squares of weight $1/2$ products also span this space.
\end{corollary}

We label the $8$ weight $1/2$ products $d_i$, $1\leq i\leq 8$. Let $v_i$ be the multiplier system of $d_i$ on $\Orth_1(U(4)\oplus U(2)\oplus A_1)$. Assume that $d_i^2$ for $1\leq i\leq 4$ are linearly independent over $\CC$. Let $\Orth_{1'}(U(4)\oplus U(2)\oplus A_1)$ be the subgroup of $\Orth_1(U(4)\oplus U(2)\oplus A_1)$ generated by all reflections associated to the divisor of $\Phi_{7,A_1(4)}/ (\prod_{j=1}^4d_j)$. The four multiplier systems $v_i$ for $1\leq i\leq 4$ coincide on  $\Orth_{1'}(U(4)\oplus U(2)\oplus A_1)$ because the quotient of any two of them is a character of $\Orth_{1'}(U(4)\oplus U(2)\oplus A_1)$ which equals one on the generator reflections of that group. Their common restriction defines a multiplier system of $\Orth_{1'}(U(4)\oplus U(2)\oplus A_1)$ and we denote it by $v$. This implies that $\Orth_{1'}(U(4)\oplus U(2)\oplus A_1)$ does not contain reflections associated to the divisor of $\prod_{j=1}^4d_j$, and that $v=1$ for all reflections in $\Orth_{1'}(U(4)\oplus U(2)\oplus A_1)$. We obtain the following result:
\begin{theorem}\label{th:projective15}
\begin{align*}
M_*(\Orth_{1'}(U(4)\oplus U(2)\oplus A_1),v) &=\CC[d_1,d_2,d_3,d_4],\\
(\cD_{3}/ \Orth_{1'}(U(4)\oplus U(2)\oplus A_1))^* &\cong \PP^3(\CC).
\end{align*}
\end{theorem}

\begin{remark}
There are four linearly independent additive theta lifts of weight $1$ for the Weil representation attached to $U(4)\oplus U(2)\oplus A_1$. This implies that every weight $1$ modular form of trivial character on $\Orth_1(U(4)\oplus U(2)\oplus A_1)$ is an additive lift and therefore is modular under the full discriminant kernel. Similarly to Remark \ref{rem:irreducible-Phi4}, we conclude that $M_*(\widetilde{\Orth}^+(U(4)\oplus U(2)\oplus A_1))$ is freely generated by the four additive lifts and that
$$
\widetilde{\Orth}^+(U(4)\oplus U(2)\oplus A_1)=\Orth_1(U(4)\oplus U(2)\oplus A_1).
$$
\end{remark}

\section{The \texorpdfstring{$U \oplus U(2) \oplus A_1(2)$}{} lattice}\label{sec:U2A1(2)}
In this section we find two interesting free algebras of modular forms for the $U \oplus U(2) \oplus A_1(2)$ lattice. We first describe some Borcherds products:
\vspace{2mm}
\begin{enumerate}
\item There are 3 holomorphic Borcherds products $t_1, t_2, t_3$ of weight $1$. They have principal parts of the form
$$
q^{-1/8}(e_v + e_{-v}), \quad (v,v) = 1/4, \; \ord(v)=4.
$$
\item There is a holomorphic Borcherds product $\Psi_{8,A_1(2)}$ of weight $8$ whose principal part with respect to the Gram matrix
$$
\begin{psmallmatrix}
0 & 0 & 0 & 0 & 1 \\ 
0 & 0 & 0 & 2 & 0 \\ 
0 & 0 & 4 & 0 & 0 \\ 
0 & 2 & 0 & 0 & 0 \\ 
1 & 0 & 0 & 0 & 0
\end{psmallmatrix}
$$
is
\begin{align*} 
&-q^{-1/8} (e_{(0, 0, 1/4, 0, 0)} + e_{(0, 0, 3/4, 0, 0)} + e_{(0, 0, 1/4, 1/2, 0)} \\ & \quad \quad \quad + e_{(0, 0, 3/4, 1/2, 0)} + e_{(0, 1/2, 1/4, 0, 0)} + e_{(0, 1/2, 3/4, 0, 0)}) \\ &+ q^{-1/2} e_{(0, 0, 1/2, 0, 0)} + q^{-1} e_{(0, 0, 0, 0, 0)}. & 
\end{align*}
\end{enumerate}

In addition, we let $m_2$ be the additive theta lift of the (unique) modular form of weight $3/2$ for the Weil representation attached to $U \oplus U(2) \oplus A_1(2)$ with constant term $1 \mathfrak{e}_0$; in particular, $m_2$ has weight two. 

Let $\Orth_1(U \oplus U(2) \oplus A_1(2))$ and $\Orth_2(U \oplus U(2) \oplus A_1(2))$ be the subgroups of $\Orth^+(U \oplus U(2) \oplus A_1(2))$ generated by the reflections associated to the divisors of $\Psi_{8,A_1(2)}$ and $\Psi_{8,A_1(2)}t_1t_2t_3$ respectively. By considering the actions of these reflections on the input forms, we see that $t_1, t_2, t_3$ are modular forms with a character of order two on $\Orth_2(U \oplus U(2) \oplus A_1(2))$, and $m_2$ is modular on $\Orth_2(U \oplus U(2) \oplus A_1(2))$ without character. Moreover, $t_1, t_2, t_3$ have trivial character on $\Orth_1(U \oplus U(2) \oplus A_1(2))$.

\begin{lemma} The Jacobian $J = J(t_1, t_2, t_3, m_2)$ equals $\Psi_{8,A_1(2)}$ up to a nonzero multiple.
\end{lemma}
\begin{proof} Using Fourier series we computed that $J$ is not identically zero (and equals $(-1/3) \Psi_{8,A_1(2)}$ up to precision $O(q, s)^{10}$). As in the previous sections $J / \Psi_{8,A_1(2)}$ is holomorphic of weight zero and therefore constant.
\end{proof}

Therefore, we can apply Theorem \ref{th:converseJacobian}:

\begin{theorem}\label{th:projective12}
\begin{align*}
M_*(\Orth_1(U \oplus U(2) \oplus A_1(2))) &=\CC[t_1, t_2, t_3, m_2],\\
M_*(\Orth_2(U \oplus U(2) \oplus A_1(2))) &=\CC[t_1^2, t_2^2, t_3^2, m_2],\\
(\cD_{3}/ \Orth_2(U \oplus U(2) \oplus A_1(2)))^* &\cong \PP^3(\CC).
\end{align*}
\end{theorem}

Furthermore, let $\Orth_{2'}(U \oplus U(2) \oplus A_1(2))$ be the subgroup of $\Orth^+(U \oplus U(2) \oplus A_1(2))$ generated by the reflections associated to the divisor of $\Psi_{8,A_1(2)}t_2t_3$. We then have another realization of $\PP^3(\CC)$.

\begin{theorem}\label{th:projective12-2}
\begin{align*}
M_*(\Orth_{2'}(U \oplus U(2) \oplus A_1(2))) &=\CC[t_1, t_2^2, t_3^2, m_2],\\
(\cD_{3}/ \Orth_{2'}(U \oplus U(2) \oplus A_1(2)))^* &\cong\PP(1,2,2,2)\cong \PP^3(\CC).
\end{align*}
\end{theorem}

\begin{remark}
The following eight groups constructed in the paper are all subgroups of $\Orth^+(2U\oplus A_1)$:
\begin{align*}
&\Orth_1(2U(4)\oplus A_1),& &\Orth_2(2U(4)\oplus A_1),& &\Orth_1(2U(2)\oplus A_1),& &\Orth_{1'}(2U(2)\oplus A_1),&\\
&\Orth_{1''}(2U(2)\oplus A_1),& &\Orth_1(2U(3)\oplus A_1),& &\Orth_1(U(4)\oplus U(2)\oplus A_1),& &\Orth_{1'}(U(4)\oplus U(2)\oplus A_1).&
\end{align*}
For any of them, the Satake--Baily--Borel compactification of modular variety is isomorphic to $\PP^3$. The decomposition into irreducibles of the square of the Jacobian of any generators corresponds to the $\Gamma$-equivalence classes of mirrors of reflections in $\Gamma$ (see Theorem \ref{th:freeJacobian}). In particular, this decomposition is an invariant of the groups up to conjugacy. We give the decomposition for the 16 reflection groups in Table \ref{tab:data}, which shows that the eight subgroups above are pairwise non-conjugate except for the possible cases $$\Orth_1(U(4) \oplus U(2) \oplus A_1) \quad \text{and} \quad \Orth_{1'}(2U(2) \oplus A_1)$$ and $$\Orth_{1'}(U(4) \oplus U(2) \oplus A_1) \quad \text{and} \quad \Orth_1(2U(4) \oplus A_1).$$
\end{remark}

\begin{remark}
The groups of type $\Orth_1(M)$ in our paper are finite index subgroups of the full integral orthogonal group $\Orth^+(M)$. This follows easily from the Margulis normal subgroup theorem  \cite{Mar91}. It can also be proved using the basic argument in the proof of \cite[Theorem 3.1]{Wan20b}.
\end{remark}

\begin{remark}
Many examples in this paper support Conjecture 5.2 in \cite{Wan20} which states that if $M_*(\Gamma)$ is a free algebra for a finite index subgroup $\Gamma$ of $\Orth^+(M)$ then $M_*(\Gamma_1)$ is also free for any other reflection subgroup $\Gamma_1$ satisfying  $\Gamma< \Gamma_1 < \Orth^+(M)$.
\end{remark}

\begin{table}[htp]
\caption{16 reflection groups $\Gamma$ for which $(\mathcal{D}/\Gamma)^*$ is a projective space. The symbol $k(J)$ stands for the weight of the Jacobian $J$. For a group $\Gamma$, the entry $\mathbf{a_1}\times x_1 + \cdots + \mathbf{a_t}\times x_t$ in the rightmost column means that the decomposition of $J^2$ into irreducibles in $M_*(\Gamma)$ consists of $x_i$ forms of weight $\mathbf{a_i}$,  $1\leq i \leq t$. }\label{tab:data}
\renewcommand\arraystretch{1.2}
\noindent\[
\begin{array}{c|c|c|c}
\text{group} & k(J) & \text{weights of generators} & \text{decomposition of $J^2$}  \\ 
\hline 
\Orth_1(2U(4)\oplus A_1) & 5 & \frac{1}{2},\frac{1}{2},\frac{1}{2},\frac{1}{2} & \mathbf{1}\times 10  \\
\hline
\Orth_2(2U(4)\oplus A_1) & 7 & 1,1,1,1 & \mathbf{1}\times 6 + \mathbf{2}\times 2 +\mathbf{4} \\
\hline
\Orth_1(2U(2)\oplus A_1) & 11 & 2,2,2,2 & \mathbf{2}\times 9 + \mathbf{4} \\
\hline
\Orth_{1'}(2U(2)\oplus A_1) & 7 & 1,1,1,1 & \mathbf{1}\times 8 +  \mathbf{2}\times 3 \\
\hline
\Orth_{1''}(2U(2)\oplus A_1) & 10 & 1,2,2,2 & \mathbf{2}\times 8 + \mathbf{4} \\
\hline
\Orth_1(2U(3)\oplus A_1) & 7 & 1,1,1,1 & \mathbf{2}+\mathbf{12} \\
\hline
\Orth_1(U(4)\oplus U(2)\oplus A_1) & 7 & 1,1,1,1 & \mathbf{1}\times 8 + \mathbf{2}\times 3\ \\
\hline
\Orth_{1'}(U(4)\oplus U(2)\oplus A_1) & 5 & \frac{1}{2},\frac{1}{2},\frac{1}{2},\frac{1}{2} & \mathbf{1}\times 10\\
\hline
\Orth_2(U\oplus U(2)\oplus A_1(2)) & 11 & 2,2,2,2 & \mathbf{2}\times 7 + \mathbf{8} \\
\hline
\Orth_{2'}(U\oplus U(2)\oplus A_1(2)) & 10 & 1,2,2,2 & \mathbf{2}\times 6 + \mathbf{8} \\
\hline
\Orth_1(2U(2)\oplus 2A_1) & 14 & 2,2,2,2,2 & \mathbf{2}\times 10 + \mathbf{8} \\
\hline
\Orth_{1'}(2U(2)\oplus 2A_1) & 9 & 1,1,1,1,1 & \mathbf{2}\times 5 + \mathbf{8} \\
\hline
\Orth_{1''}(2U(2)\oplus 2A_1) & 13 & 1,2,2,2,2 & \mathbf{2}\times 9 + \mathbf{8} \\
\hline
\Orth_1(2U(3)\oplus A_2) & 9 & 1,1,1,1,1 & \mathbf{18}  \\
\hline
\Orth_{1,1234}(2U(3)\oplus A_2) & 13 & 1,2,2,2,2 & \mathbf{2}\times 4 + \mathbf{18}  \\
\hline
\Orth_{2}(2U(3)\oplus A_2) & 14 & 2,2,2,2,2 & \mathbf{2}\times 5 + \mathbf{18}  \\
\hline
\end{array} 
\]
\end{table}

\bigskip

\noindent
\textbf{Acknowledgements} 
The authors would like to thank Eberhard Freitag and Riccardo Salvati Manni for proposing this topic and for helpful discussions. H. Wang is grateful to Max Planck Institute for Mathematics in Bonn for its hospitality and financial support. B. Williams is supported by a fellowship of the LOEWE research group Uniformized Structures in Algebra and Geometry.
 
\clearpage
\newpage

\section*{Appendix A: Type II products and \texorpdfstring{$*$}{}-sets of type \texorpdfstring{$2U(4)\oplus A_1$}{}}

The sixty type II singular-weight products attached to the lattice $M = 2U(4)\oplus A_1$ have input forms with principal part $q^{-1/4}(e_v + e_{-v})$ for certain cosets $v \in M^{\vee} / M$ of order 4. In the table below, we list one coset representative $v$ for each product $\Theta_i$, with respect to the Gram matrix $$\begin{psmallmatrix} 0 & 0 & 0 & 0 & 4 \\ 0 & 0 & 0 & 4 & 0 \\ 0 & 0 & 2 & 0 & 0 \\ 0 & 4 & 0 & 0 & 0 \\ 4 & 0 & 0 & 0 & 0 \end{psmallmatrix}:$$

\begin{table}[htp]
\begin{tabular}{lllll}
$\Theta_{1}$ & $(1/4, 1/2, 0, 1/2, 1/4)$ & & $\Theta_{2}$ & $(1/4, 0, 0, 0, 1/4)$ \\
$\Theta_{3}$ & $(0, 0, 1/2, 1/4, 1/4)$ & & $\Theta_{4}$ & $(1/2, 1/2, 1/2, 1/4, 1/4)$ \\
$\Theta_{5}$ & $(1/4, 0, 0, 1/4, 1/4)$ & & $\Theta_{6}$ & $(1/4, 1/2, 0, 1/4, 3/4)$ \\
$\Theta_{7}$ & $(1/2, 0, 1/2, 1/4, 0)$ & & $\Theta_{8}$ & $(0, 0, 1/2, 1/4, 1/2)$ \\
$\Theta_{9}$ & $(0, 1/4, 1/2, 0, 1/4)$ & & $\Theta_{10}$ & $(1/2, 1/4, 1/2, 1/2, 3/4)$ \\
$\Theta_{11}$ & $(1/4, 1/2, 1/2, 0, 0)$ & & $\Theta_{12}$ & $(1/4, 0, 1/2, 0, 0)$ \\
$\Theta_{13}$ & $(1/4, 1/2, 1/2, 1/4, 1/2)$ & & $\Theta_{14}$ & $(1/4, 1/2, 1/2, 3/4, 1/2)$ \\
$\Theta_{15}$ & $(1/4, 3/4, 1/2, 0, 0)$ & & $\Theta_{16}$ & $(1/4, 1/4, 1/2, 0, 0)$ \\
$\Theta_{17}$ & $(1/4, 3/4, 0, 1/2, 3/4)$ & & $\Theta_{18}$ & $(1/4, 3/4, 0, 0, 1/4)$ \\
$\Theta_{19}$ & $(0, 1/4, 0, 1/4, 1/4)$ & & $\Theta_{20}$ & $(1/2, 1/4, 0, 3/4, 1/4)$ \\
$\Theta_{21}$ & $(1/4, 3/4, 0, 1/4, 1/2)$ & & $\Theta_{22}$ & $(1/4, 3/4, 0, 3/4, 0)$ \\
$\Theta_{23}$ & $(0, 1/2, 1/2, 1/2, 1/4)$ & & $\Theta_{24}$ & $(0, 0, 1/2, 1/2, 1/4)$ \\
$\Theta_{25}$ & $(1/2, 1/4, 0, 1/4, 0)$ & & $\Theta_{26}$ & $(1/2, 1/4, 0, 1/4, 1/2)$ \\
$\Theta_{27}$ & $(1/2, 1/4, 1/2, 0, 0)$ & & $\Theta_{28}$ & $(1/2, 1/4, 1/2, 0, 1/2)$ \\
$\Theta_{29}$ & $(1/4, 1/4, 1/2, 1/4, 3/4)$ & & $\Theta_{30}$ & $(1/4, 1/4, 1/2, 3/4, 1/4)$ \\
$\Theta_{31}$ & $(1/4, 0, 0, 1/2, 1/4)$ & & $\Theta_{32}$ & $(1/4, 1/2, 0, 0, 1/4)$ \\
$\Theta_{33}$ & $(0, 0, 1/2, 1/4, 3/4)$ & & $\Theta_{34}$ & $(1/2, 1/2, 1/2, 1/4, 3/4)$ \\
$\Theta_{35}$ & $(1/4, 1/2, 0, 3/4, 3/4)$ & & $\Theta_{36}$ & $(1/4, 0, 0, 3/4, 1/4)$ \\
$\Theta_{37}$ & $(1/2, 0, 1/2, 1/4, 1/2)$ & & $\Theta_{38}$ & $(0, 0, 1/2, 1/4, 0)$ \\
$\Theta_{39}$ & $(1/2, 1/4, 1/2, 1/2, 1/4)$ & & $\Theta_{40}$ & $(0, 1/4, 1/2, 0, 3/4)$ \\
$\Theta_{41}$ & $(1/4, 0, 1/2, 1/2, 0)$ & & $\Theta_{42}$ & $(1/4, 1/2, 1/2, 1/2, 0)$ \\
$\Theta_{43}$ & $(1/4, 0, 1/2, 1/4, 0)$ & & $\Theta_{44}$ & $(1/4, 0, 1/2, 3/4, 0)$ \\
$\Theta_{45}$ & $(1/4, 1/4, 1/2, 1/2, 1/2)$ & & $\Theta_{46}$ & $(1/4, 3/4, 1/2, 1/2, 1/2)$ \\
$\Theta_{47}$ & $(1/4, 1/4, 0, 1/2, 3/4)$ & & $\Theta_{48}$ & $(1/4, 1/4, 0, 0, 1/4)$ \\
$\Theta_{49}$ & $(1/2, 1/4, 0, 3/4, 3/4)$ & & $\Theta_{50}$ & $(0, 1/4, 0, 1/4, 3/4)$ \\
$\Theta_{51}$ & $(1/4, 1/4, 0, 3/4, 1/2)$ & & $\Theta_{52}$ & $(1/4, 1/4, 0, 1/4, 0)$ \\
$\Theta_{53}$ & $(0, 1/2, 1/2, 0, 1/4)$ & & $\Theta_{54}$ & $(0, 0, 1/2, 0, 1/4)$ \\
$\Theta_{55}$ & $(0, 1/4, 0, 1/4, 0)$ & & $\Theta_{56}$ & $(0, 1/4, 0, 1/4, 1/2)$ \\
$\Theta_{57}$ & $(0, 1/4, 1/2, 0, 1/2)$ & & $\Theta_{58}$ & $(0, 1/4, 1/2, 0, 0)$ \\
$\Theta_{59}$ & $(1/4, 3/4, 1/2, 1/4, 1/4)$ & & $\Theta_{60}$ & $(1/4, 3/4, 1/2, 3/4, 3/4)$ \\
\end{tabular}\end{table}

\begin{figure}[htp]
\includegraphics[width=0.75\linewidth]{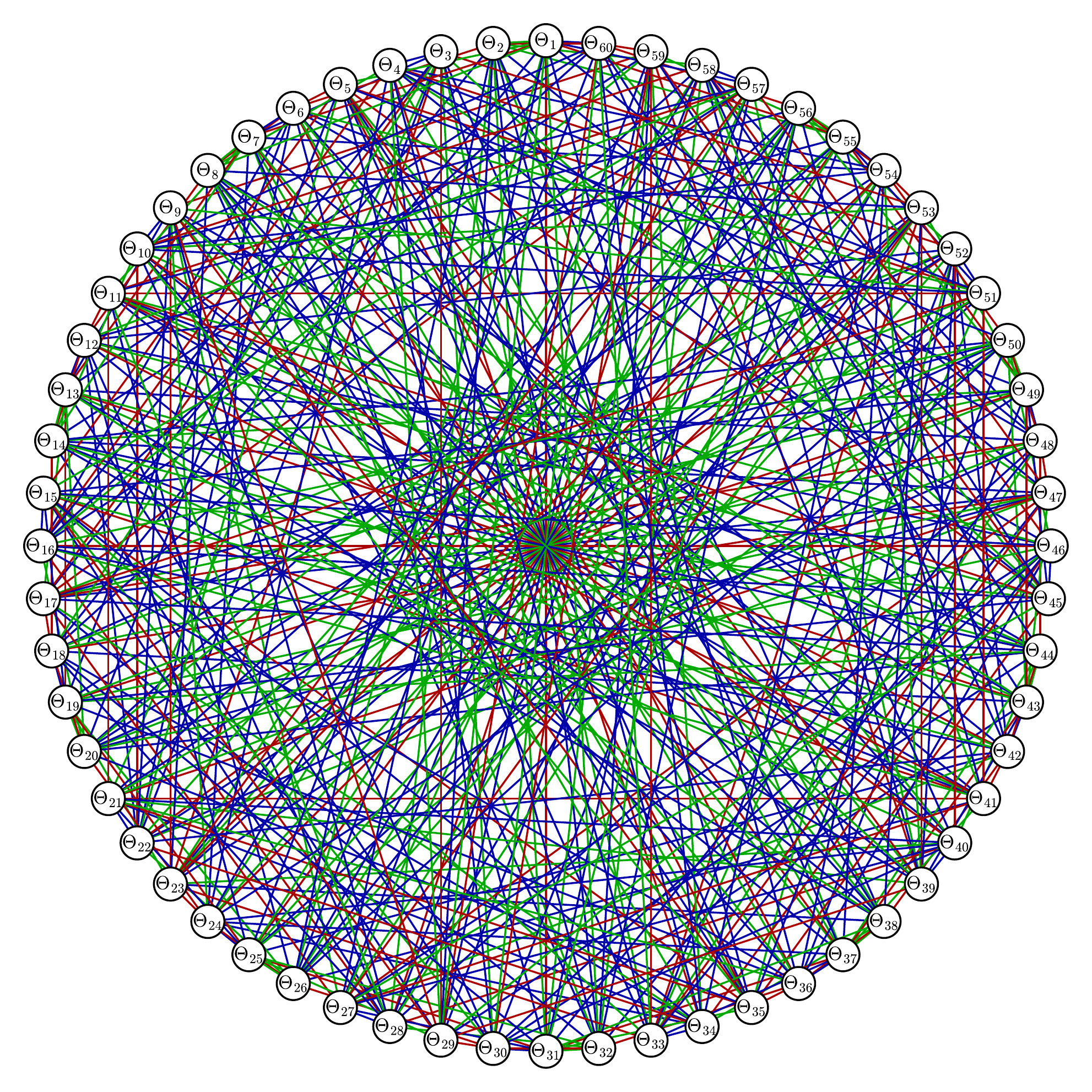}
\caption{There is an edge between the type II products $\Theta_i$ and $\Theta_j$ if $\Theta_i$ is modular under the reflections associated to the divisor of $\Theta_j$. This is a 15-regular graph with 23,040 automorphisms. The edge between $\Theta_i$ and $\Theta_j$ is colored red, green or blue depending on $i+j$ mod $3$.} \label{fig:A1-1}
\end{figure}

The $*$-sets are the maximum cliques in the graph formed by connecting products $\Theta_i$ and $\Theta_j$ by an edge if the reflections associated to the divisor of $\Theta_j$ also preserve the divisor of $\Theta_i$. With respect to the ordering above, the 105 $*$-sets are $\{\Theta_{i_1},\Theta_{i_2}, \Theta_{i_3}, \Theta_{i_4}\}$ where $(i_1,i_2,i_3,i_4)$ is one of the following: 

\begingroup 
\allowdisplaybreaks
\begin{small}
\begin{align*} 
&(1, 2, 26, 55), (3, 17, 18, 34), (6, 19, 36, 49), (10, 14, 40, 44), (13, 19, 20, 44), (17, 27, 47, 57), (23, 24, 53, 54),\\
&(1, 2, 29, 60), (3, 21, 22, 34), (7, 8, 37, 38), (10, 30, 40, 60), (13, 30, 43, 60), (17, 48, 49, 50), (23, 24, 57, 58),\\
&(1, 2, 31, 32), (3, 23, 33, 53), (7, 11, 38, 42), (10, 39, 53, 54), (14, 18, 43, 47), (18, 19, 20, 47), (23, 54, 55, 56),\\
&(1, 5, 32, 36), (4, 8, 34, 38), (7, 13, 14, 38), (10, 39, 57, 58), (14, 29, 44, 59), (18, 28, 48, 58), (24, 25, 26, 53),\\
&(1, 7, 8, 32), (4, 24, 34, 54), (7, 24, 37, 54), (11, 12, 28, 57), (14, 43, 49, 50), (19, 20, 49, 50), (25, 26, 55, 56),\\
&(1, 18, 31, 48), (4, 33, 47, 48), (8, 12, 37, 41), (11, 12, 41, 42), (15, 16, 28, 57), (19, 23, 50, 54), (25, 29, 56, 60),\\
&(1, 27, 31, 57), (4, 33, 51, 52), (8, 23, 38, 53), (11, 12, 45, 46), (15, 16, 41, 42), (19, 25, 26, 50), (25, 31, 32, 56),\\
&(2, 6, 31, 35), (5, 6, 22, 51), (8, 37, 43, 44), (11, 21, 41, 51), (15, 16, 45, 46), (20, 24, 49, 53), (26, 30, 55, 59),\\
&(2, 17, 32, 47), (5, 6, 35, 36), (9, 10, 22, 51), (11, 26, 41, 56), (15, 19, 45, 49), (20, 49, 55, 56), (27, 28, 53, 54),\\
&(2, 28, 32, 58), (5, 6, 39, 40), (9, 10, 35, 36), (11, 42, 43, 44), (15, 29, 30, 46), (21, 22, 47, 48), (27, 28, 57, 58),\\
&(2, 31, 37, 38), (5, 15, 35, 45), (9, 10, 39, 40), (12, 13, 14, 41), (15, 33, 34, 46), (21, 22, 51, 52), (27, 41, 42, 58),\\
&(3, 4, 16, 45), (5, 20, 35, 50), (9, 13, 39, 43), (12, 22, 42, 52), (16, 20, 46, 50), (21, 25, 51, 55), (27, 45, 46, 58),\\
&(3, 4, 29, 30), (5, 36, 37, 38), (9, 23, 24, 40), (12, 25, 42, 55), (16, 45, 59, 60), (21, 35, 36, 52), (29, 30, 59, 60),\\
&(3, 4, 33, 34), (6, 7, 8, 35), (9, 27, 28, 40), (13, 14, 43, 44), (17, 18, 47, 48), (21, 39, 40, 52), (30, 31, 32, 59),\\
&(3, 7, 33, 37), (6, 16, 36, 46), (9, 29, 39, 59), (13, 17, 44, 48), (17, 18, 51, 52), (22, 26, 52, 56), (33, 34, 59, 60).
\end{align*}
\end{small}
\endgroup

The ordering is chosen such that the 90 exceptional pairs of type II forms that extend in three ways to $*$-sets are precisely those of the form $(\Theta_{i+j}, \Theta_{i+k})$ where $i < 30$ is odd and $j, k \in \{0, 1, 30, 31\}$ are distinct.

\begin{figure}[htp]
\includegraphics[width=0.75\linewidth]{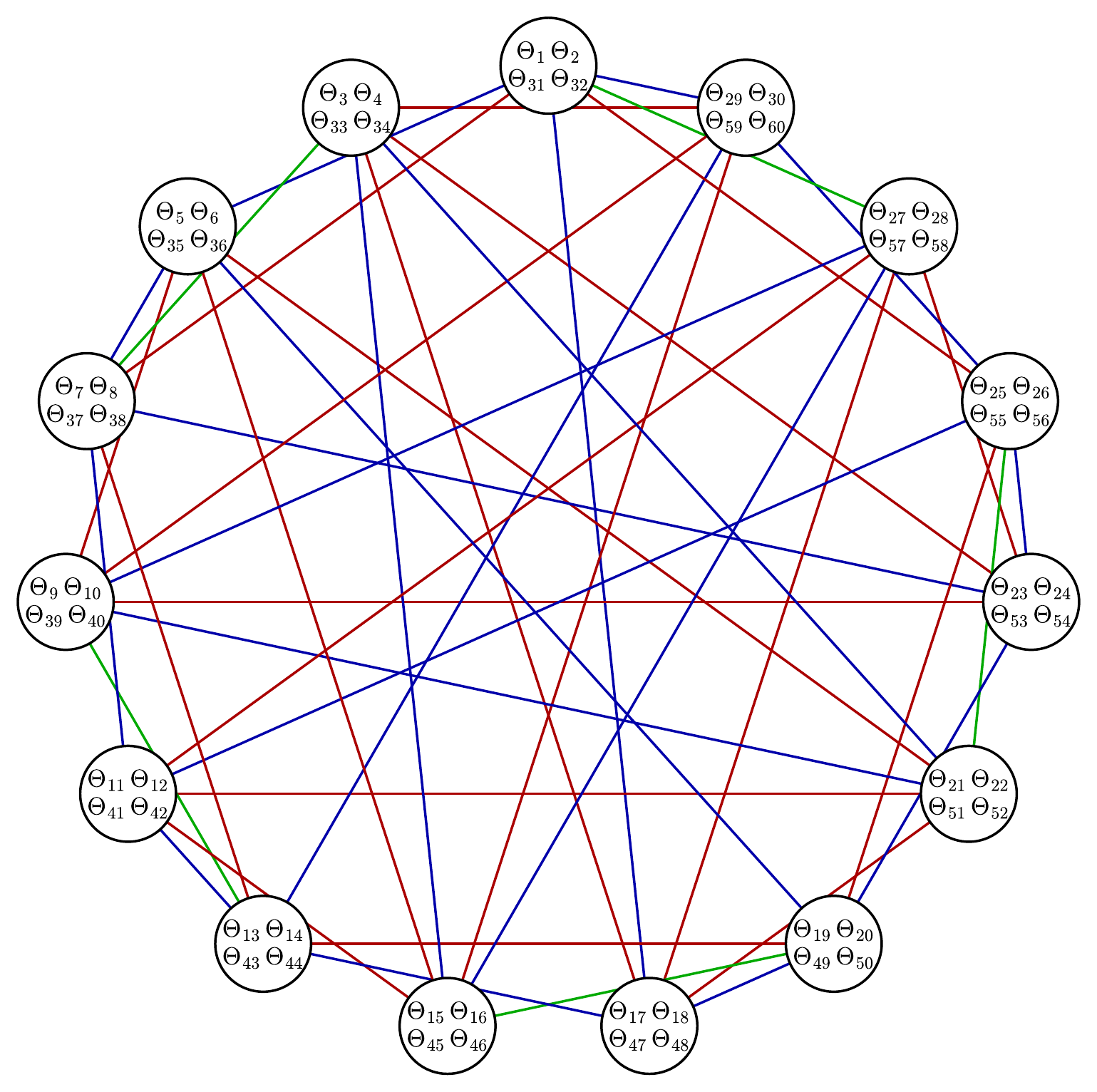}
\caption{Contracting the 90 exceptional pairs in Figure 1 yields the strongly regular graph $\mathrm{srg}(15, 6, 1, 3)$ on 15 vertices. The edges are colored red, green, blue by the same rule as Figure \ref{fig:A1-1}.}
\end{figure}

\newpage

\section*{Appendix B: Singular-weight products and \texorpdfstring{$*$}{}-sets of type \texorpdfstring{$2U(3)\oplus A_2$}{}}

The $45$ singular-weight products attached to the lattice $M = 2U(3)\oplus  A_2$ have input forms with principal part $q^{-1/3}(e_v + e_{-v})$ for cosets $v \in M^{\vee} / M$ of order $3$. In the table below, for each product $G_i$ we list one coset representative $v$ with respect to the Gram matrix $$\begin{psmallmatrix} 0 & 0 & 0 & 0 & 0 & 3 \\ 0 & 0 & 0 & 0 & 3 & 0 \\ 0 & 0 & 2 & 1 & 0 & 0 \\ 0 & 0 & 1 & 2 & 0 & 0 \\ 0 & 3 & 0 & 0 & 0 & 0 \\ 3 & 0 & 0 & 0 & 0 & 0 \end{psmallmatrix}:$$

\begin{table}[htp]
\begin{tabular}{lllll}
$G_{1}$ & $(0, 0, 2/3, 2/3, 0, 1/3)$ & & $G_{2}$ & $(0, 0, 2/3, 2/3, 1/3, 1/3)$ \\
$G_{3}$ & $(0, 1/3, 2/3, 2/3, 0, 0)$ & & $G_{4}$ & $(1/3, 2/3, 1/3, 1/3, 1/3, 1/3)$ \\
$G_{5}$ & $(0, 1/3, 1/3, 1/3, 0, 1/3)$ & & $G_{6}$ & $(0, 0, 2/3, 2/3, 2/3, 1/3)$ \\
$G_{7}$ & $(0, 1/3, 1/3, 1/3, 0, 0)$ & & $G_{8}$ & $(1/3, 0, 1/3, 1/3, 2/3, 0)$ \\
$G_{9}$ & $(1/3, 1/3, 1/3, 1/3, 0, 0)$ & & $G_{10}$ & $(1/3, 1/3, 0, 0, 0, 1/3)$ \\
$G_{11}$ & $(0, 0, 2/3, 2/3, 0, 2/3)$ & & $G_{12}$ & $(1/3, 2/3, 1/3, 1/3, 2/3, 2/3)$ \\
$G_{13}$ & $(1/3, 2/3, 0, 0, 1/3, 2/3)$ & & $G_{14}$ & $(1/3, 0, 2/3, 2/3, 1/3, 0)$ \\
$G_{15}$ & $(1/3, 1/3, 2/3, 2/3, 0, 0)$ & & $G_{16}$ & $(0, 0, 2/3, 2/3, 2/3, 2/3)$ \\
$G_{17}$ & $(1/3, 1/3, 0, 0, 1/3, 0)$ & & $G_{18}$ & $(1/3, 2/3, 1/3, 1/3, 0, 0)$ \\
$G_{19}$ & $(0, 1/3, 1/3, 1/3, 0, 2/3)$ & & $G_{20}$ & $(0, 1/3, 2/3, 2/3, 0, 2/3)$ \\
$G_{21}$ & $(0, 0, 2/3, 2/3, 1/3, 0)$ & & $G_{22}$ & $(0, 1/3, 0, 0, 1/3, 1/3)$ \\
$G_{23}$ & $(0, 1/3, 0, 0, 1/3, 2/3)$ & & $G_{24}$ & $(0, 0, 2/3, 2/3, 0, 0)$ \\
$G_{25}$ & $(1/3, 2/3, 0, 0, 0, 1/3)$ & & $G_{26}$ & $(0, 1/3, 2/3, 2/3, 0, 1/3)$ \\
$G_{27}$ & $(1/3, 2/3, 2/3, 2/3, 2/3, 2/3)$ & & $G_{28}$ & $(1/3, 0, 0, 0, 2/3, 1/3)$ \\
$G_{29}$ & $(1/3, 1/3, 1/3, 1/3, 1/3, 2/3)$ & & $G_{30}$ & $(1/3, 0, 0, 0, 0, 1/3)$ \\
$G_{31}$ & $(0, 0, 2/3, 2/3, 1/3, 2/3)$ & & $G_{32}$ & $(1/3, 0, 0, 0, 1/3, 1/3)$ \\
$G_{33}$ & $(1/3, 2/3, 2/3, 2/3, 0, 0)$ & & $G_{34}$ & $(1/3, 0, 1/3, 1/3, 1/3, 0)$ \\
$G_{35}$ & $(0, 1/3, 0, 0, 1/3, 0)$ & & $G_{36}$ & $(1/3, 2/3, 0, 0, 2/3, 0)$ \\
$G_{37}$ & $(1/3, 1/3, 0, 0, 2/3, 2/3)$ & & $G_{38}$ & $(1/3, 0, 1/3, 1/3, 0, 0)$ \\
$G_{39}$ & $(1/3, 1/3, 2/3, 2/3, 1/3, 2/3)$ & & $G_{40}$ & $(1/3, 0, 2/3, 2/3, 0, 0)$ \\
$G_{41}$ & $(0, 0, 2/3, 2/3, 2/3, 0)$ & & $G_{42}$ & $(1/3, 2/3, 2/3, 2/3, 1/3, 1/3)$ \\
$G_{43}$ & $(1/3, 1/3, 2/3, 2/3, 2/3, 1/3)$ & & $G_{44}$ & $(1/3, 0, 2/3, 2/3, 2/3, 0)$ \\
$G_{45}$ & $(1/3, 1/3, 1/3, 1/3, 2/3, 1/3)$
\end{tabular}\end{table}

With respect to this ordering the 27 $*$-sets are the sets $\{G_i, \, i \in I\}$ where $I$ is one of the index sets:

\begin{align*}
&(1, 14, 15, 23, 27), (1, 22, 33, 39, 44), (1, 35, 40, 42, 43), (2, 3, 14, 25, 45), (2, 9, 13, 20, 44), \\ 
& (2, 26, 29, 36, 40), (3, 8, 10, 31, 42), (3, 12, 21, 30, 39), (4, 6, 7, 10, 44), (4, 11, 35, 38, 45), \\
&(4, 15, 21, 26, 28), (5, 6, 14, 18, 37), (5, 9, 28, 41, 42), (5, 16, 36, 38, 39), (6, 12, 17, 19, 40),\\
& (7, 16, 25, 34, 43), (7, 27, 29, 30, 41), (8, 11, 18, 22, 29), (8, 13, 15, 16, 19), (9, 11, 12, 23, 34), \\
&(10, 23, 24, 32, 36), (13, 24, 30, 35, 37), (17, 20, 27, 31, 38), (17, 22, 24, 25, 28), (18, 20, 21, 32, 43),\\
& (19, 32, 33, 41, 45), (26, 31, 33, 34, 37).
\end{align*} 
These can be computed as the maximal cliques in the graph depicted in Figure \ref{fig:A2}.

\clearpage
\begin{figure}[htp]
\includegraphics[width=0.75\linewidth]{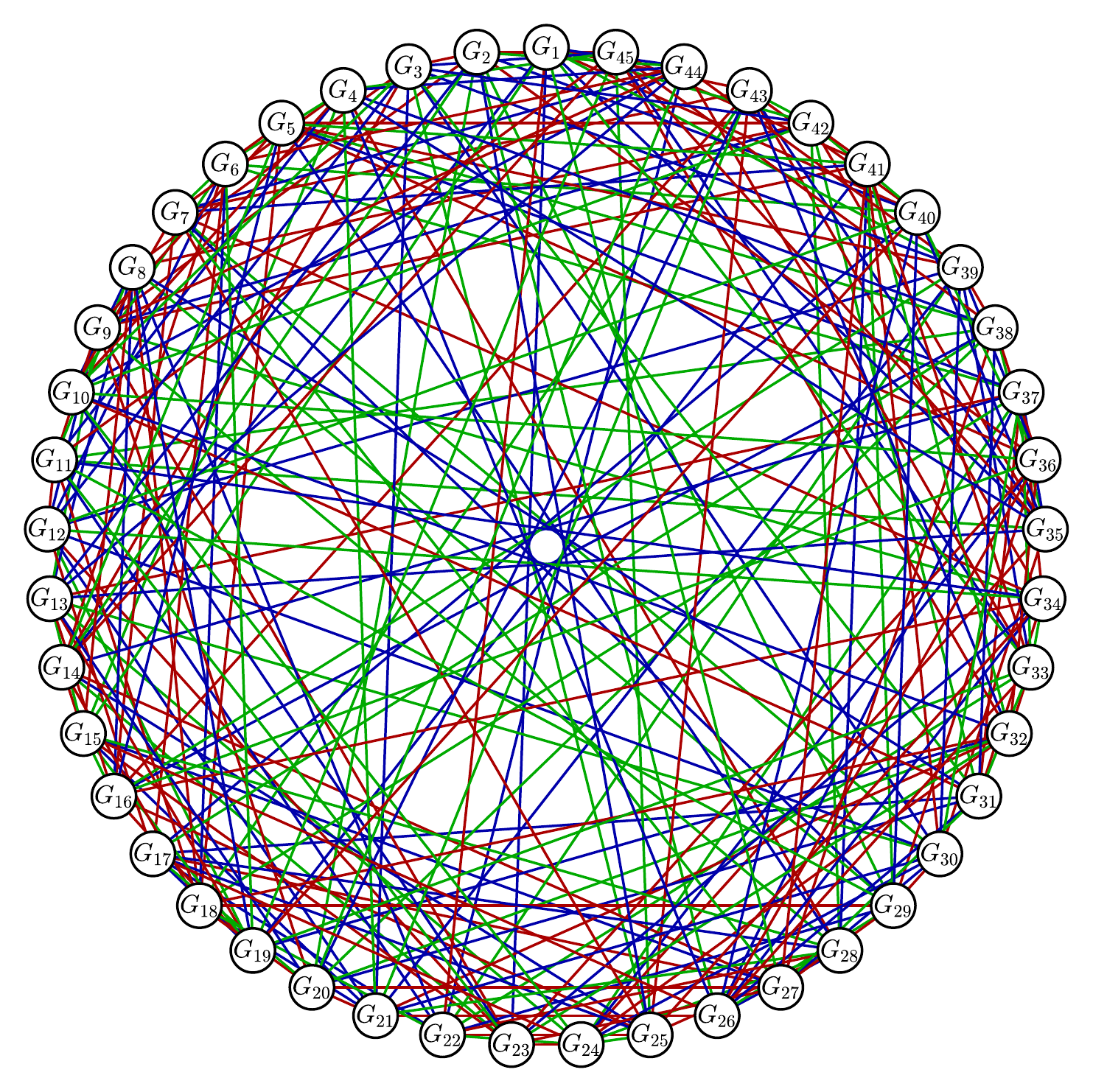}
\caption{There is an edge between the singular-weight products $G_i$ and $G_j$ if $G_i$ is modular under the reflections associated to the divisor of $G_j$. This is a strongly regular graph with parameters $(45, 12, 3, 3)$ and has 51,840 automorphisms.} \label{fig:A2}
\end{figure}

\bibliographystyle{plainnat}
\bibliofont
\bibliography{bibtex}

\end{document}